\documentclass[12pt]{amsart}
\usepackage{geometry}
\usepackage{hyperref}
                \usepackage{amsthm,amsfonts,amsmath,amscd,amssymb,epsfig,verbatim,
}

\usepackage{graphicx}
\usepackage{epstopdf}
\title{Flexible Weinstein manifolds}
\author{Kai~Cieliebak and Yasha~Eliashberg} 
\address{Kai Cieliebak \\ Institut f\"ur Mathematik \\
    Universit\"at Augsburg \\
 Germany}
 \address{Yasha Eliashberg \\ Department of Mathematics \\   Stanford University \\ USA}
 \thanks{The second author was partially supported by an NSF grant
 DMS-1205349.}
 \date{}  
\let\oldmarginpar\marginpar
\renewcommand\marginpar[1]{\-\oldmarginpar[\raggedleft\footnotesize #1]%
{\raggedright\footnotesize #1}}

\theoremstyle{plain}
\newtheorem{theorem}{Theorem}[section]
\newtheorem{thm}[theorem]{Theorem}

\newtheorem{cor}[theorem]{Corollary}
\newtheorem{proposition}[theorem]{Proposition}
\newtheorem{prop}[theorem]{Proposition}
\newtheorem{lemma}[theorem]{Lemma}

\theoremstyle{remark}
\newtheorem*{conjecture}{Conjecture}

\newtheorem{remark}[theorem]{Remark}
\newtheorem*{remark*}{Remark}

\newtheorem*{example*}{Example}

\newtheoremstyle{step}
{3pt}
{3pt}
{}
{\parindent}
{\bf}
{.}
{.5em}
{}
\theoremstyle{step}
\newtheorem*{step1}{Step 1}
\newtheorem*{step2}{Step 2}
\newtheorem*{step3}{Step 3}

\theoremstyle{definition}
\newtheorem*{definition}{Definition}
\newcommand{\id}{{{\mathchoice {\rm 1\mskip-4mu l} {\rm 1\mskip-4mu l}
{\rm 1\mskip-4.5mu l} {\rm 1\mskip-5mu l}}}}

\newcommand{\wt}{\widetilde}
\newcommand{\wh}{\widehat}
\newcommand{\ol}{\overline}

\newcommand{\p}{\partial}

\newcommand{\om}{\omega}

\newcommand{\eps}{\varepsilon}
\newcommand{\into}{\hookrightarrow}

\newcommand{\N}{{\mathbb{N}}}

\newcommand{\Z}{{\mathbb{Z}}}
\newcommand{\R}{{\mathbb{R}}}
\newcommand{\C}{{\mathbb{C}}}

\newcommand{\bS}{{\bf S}}

\newcommand{\st}{{\rm st}}

\newcommand{\Int}{{\rm Int\,}} 

\renewcommand{\min}{{\rm min}}
\renewcommand{\max}{{\rm max}}

\newcommand{\tb}{{\rm tb}}

\newcommand{\Id}{\mathrm {Id}}

\newcommand{\Diff}{\mathrm{Diff}}

\newcommand{\flex}{\mathrm{flex}}
\newcommand{\Skel}{\mathrm{Skel}}

\newcommand{\Stein}{\mathfrak{Stein}}
\newcommand{\Weinstein}{\mathfrak{Weinstein}}

\newcommand{\Smale}{\mathfrak{Smale}}
\newcommand{\Morse}{\mathfrak{Morse}}
\newcommand{\Liouville}{\mathfrak{Liouville}}

\newcommand{\EE}{\mathcal{E}}
\newcommand{\DD}{\mathcal{D}}

\newcommand{\PP}{\mathcal{P}}

\newcommand{\XX}{\mathcal{X}}

\newcommand{\UU}{\mathcal{U}}
\newcommand{\WW}{\mathcal{W}}

\def\Op{{\mathcal O}{\it p}\,}

\newcommand{\fS}{\mathfrak{S}}
\newcommand{\fM}{\mathfrak{M}}

\newcommand{\fW}{{\mathfrak W}}

\numberwithin{figure}{section}

\begin{document}

\begin{abstract}
This survey on flexible Weinstein manifolds is, essentially, an
extract from the book~\cite{CieEli12}.
\end{abstract}

\dedicatory{To Alan Weinstein with admiration.}

\maketitle
\section{Introduction}\label{sec:intro}

The notion of a {\it Weinstein manifold} was introduced 
in~\cite{EliGro91}, formalizing the symplectic handlebody construction
from Alan Weinstein's paper~\cite{Wei91} and the Stein handlebody
construction from~\cite{Eli90}. Since then, the notion of a Weinstein
manifold has become one of the central notions in
symplectic and contact topology. The existence question for Weinstein
structures on manifolds of dimension $>4$ was settled
in~\cite{Eli90}. The past five years have brought two major breakthroughs
on the uniqueness question: By work of McLean~\cite{McL09} and others
we now know that, on any manifold of dimension $>4$ which admits a
Weinstein structure, there exist infinitely many Weinstein structures
that are pairwise non-homotopic (but formally homotopic). On the other
hand, Murphy's $h$-principle for loose Legendrian knots~\cite{Mur11}
has led to the notion of  {\it flexible} Weinstein structures, which are
unique up to homotopy in their formal class. In this survey, which is
essentially 
an extract from the book~\cite{CieEli12}, we discuss this uniqueness
result and some of its applications. 

\subsection{Weinstein manifolds and cobordisms}
 
\begin{definition}
A {\em Weinstein structure} on an open manifold $V$ is a triple
$(\om,X,\phi)$, where 
\begin{itemize}
\item $\om$ is a symplectic form on $V$,
\item $\phi:V\to\R$ is an exhausting generalized Morse function,
\item $X$ is a complete vector field which is Liouville for $\om$ and
  gradient-like for $\phi$. 
\end{itemize} 
The quadruple $(V,\om,X,\phi)$ is then called a {\em Weinstein
  manifold}. 
\end{definition}

Let us explain all the terms in this definition. A {\em symplectic
form} is a nondegenerate closed $2$-form $\om$. A {\em Liouville
field} for $\om$ is a vector field $X$ satisfying $L_X\om=\om$; by
Cartan's formula, this is equivalent to saying that the associated
{\em Liouville form} 
$$
   \lambda := i_X\om
$$
satisfies $d\lambda=\om$. A function $\phi:V\to\R$ is called {\em
  exhausting} if it is proper (i.e., preimages of compact sets are
compact) and bounded from below. It is called {\em Morse} if all its
critical points are nondegenerate, and {\em generalized Morse} if its
critical points are either nondegenerate or {\em embryonic},
where the latter condition means that in some local coordinates
$x_1,\dots,x_m$ near the critical point $p$ the function looks like
the function $\phi_0$ in the {\em birth--death family}
$$
   \phi_t(x) = \phi_t(p) \pm tx_1 +  x_1^3 - \sum_{i=2}^kx_i^2 +
   \sum_{j=k+1}^mx_j^2. 
$$
A vector field $X$ is called {\em complete} if its flow exists for all
times. It is called {\em gradient-like} for a function $\phi$ if 
$$
   d\phi(X) \geq \delta(|X|^2+|d\phi|^2), 
$$ 
where $\delta:V\to\R_+$ is a positive function and the norms are taken
with respect to any Riemannian metric on $V$. Note that away from
critical points this just means $d\phi(X)>0$. Critical points $p$ of
$\phi$ agree with zeroes of $X$, and $p$ is nondegenerate
(resp.~embryonic) as a critical point of $\phi$ iff it is
nondegenerate (resp.~embryonic) as a zero of $X$. 
Here a zero $p$ of a vector field $X$ is called embryonic if $X$
agrees near $p$, up to higher order terms, with the gradient of a
function having $p$ as an embryonic critical point. 

It is not hard to see that any Weinstein structure $(\om,X,\phi)$ can
be perturbed to make the function $\phi$ Morse. However, in
$1$-parameter families of Weinstein structures  embryonic zeroes are
generically unavoidable. Since we wish to study such families, we
allow for embryonic zeroes in the definition of a Weinstein
structure. 
\medskip
  
We will also consider Weinstein structures on a {\em cobordism}, i.e.,
a compact manifold $W$ with boundary $\p W=\p_+W\amalg\p_-W$. The
definition of a {\em Weinstein cobordism} $(W,\om,X,\phi)$ differs
from that of a Weinstein manifold only in replacing the condition that
$\phi$ is exhausting by the requirement that $\p_\pm W$ are regular
level sets of $\phi$ with $\phi|_{\p_-W}=\min\,\phi$ and
$\phi|_{\p_+W}=\max\,\phi$, and completeness of $X$ by the condition
that $X$ points inward along $\p_-W$ and outward along $\p_+W$.  

A Weinstein cobordism with $\p_-W=\emptyset$ is called a {\em
Weinstein domain}. Thus any Weinstein manifold $(V,\omega,
X,\phi)$ can be exhausted by Weinstein domains $W_k=\{\phi\leq c_k\}$,
where $c_k\nearrow\infty$ is a sequence of regular values of the
function $\phi$. 


The Liouville form $\lambda=i_X\om$ induces contact forms
$\alpha_c:=\lambda|_{\Sigma_c}$ and contact structures
$\xi_c:=\ker(\alpha_c)$ on all regular level sets
$\Sigma_c:=\phi^{-1}(c)$ of $\phi$. In particular, the boundary
components of a Weinstein cobordism carry contact forms which make
$\p_+W$ a  symplectically convex and 
$\p_-W$ a  symplectically concave boundary (i.e., the orientation induced by the
contact form agrees with the boundary orientation on $\p_+W$ and is
opposite to it on $\p_-W$). Contact manifolds which appear as
boundaries of Weinstein domains are called {\em Weinstein fillable}. 

A Weinstein manifold $(V,\omega,X,\phi)$ is said to be of {\it finite
  type} if $\phi$ has only finitely many critical points. By attaching 
a cylindrical end $\bigl(\R_+\times\p W,d(e^r\lambda|_{\p
  W}),\frac{\p}{\p r},f(r)\bigr)$
(i.e., the positive half of the symplectization of the contact
structure on the boundary) to the boundary, any Weinstein domain
$(W,\om,X,\phi)$ can be completed to a finite type Weinstein manifold,
called its {\em completion}. Conversely, any finite type Weinstein
manifold can be obtained by attaching a cylindrical end to a Weinstein
domain.   
\medskip

Here are some basic examples of Weinstein manifolds. 

(1) $\C^n$ with complex coordinates $x_j+iy_j$ carries the canonical
Weinstein structure 
$$
   \Bigl(\sum_jdx_j\wedge dy_j,\frac{1}{2}\sum_j(x_j\frac{\p}{\p x_j}
   + y_j\frac{\p}{\p y_j}),\sum_j(x_j^2+y_j^2)\Bigr).
$$ 
(2) The cotangent bundle $T^*Q$ of a closed manifold $Q$ carries a
canonical Weinstein structure which in canonical local coordinates 
$(q_j,p_j)$ is given by 
$$
   \Bigl(\sum_jdp_j\wedge dq_j,\sum_jp_j\frac{\p}{\p
     p_j},\sum_jp_j^2\Bigr).
$$ 
(As it stands, this is not yet a Weinstein structure
because $\sum_jp_j^2$ is not a generalized Morse function, but it
can be easily perturbed to make the function Morse.) 

(3) The product of two Weinstein manifolds $(V_1,\omega_1,X_1,\phi_1)$
and $(V_2,\omega_2,X_2,\phi_2)$ has a canonical Weinstein structure
$(V_1\times V_2,\omega_1\oplus\omega_2, X_1\oplus
X_2,\phi_1\oplus\phi_2)$. The product $V\times\C$ with its canonical
Weinstein structure is called the {\em stabilization} of the Weinstein
manifold $(V,\omega,X,\phi)$. 
\medskip

In a Weinstein manifold $(V,\om,X,\phi)$, there is an intriguing
interplay between Morse theoretic properties of $\phi$ and 
symplectic geometry: the stable manifold $W_p^-$ (with respect to the
vector field $X$) of a critical point $p$ is {\em isotropic} in the
symplectic sense (i.e., $\om|_{W_p^-}=0$), and its intersection with
every regular level set $\phi^{-1}(c)$ is {\em isotropic} in the
contact sense (i.e., it is tangent to $\xi_c$). In particular, the
Morse indices of critical points of $\phi$ are $\leq\frac12\dim V$.

\subsection{Stein -- Weinstein -- Morse}

Weinstein structures are related to several other interesting
structures as shown in the following diagram:
\begin{equation*}
\begin{CD}
   \Stein @>\fW>> \Weinstein @>\fM>> \Morse \\  
   && @VVV \\
   && \Liouville.
\end{CD}
\end{equation*}
Here $\Weinstein$ denotes the space of Weinstein structures and
$\Morse$ the space of generalized Morse functions on a fixed manifold
$V$ or a cobordism $W$. As before, we require the function $\phi$ to
be exhausting in the manifold case, and to have $\p_\pm W$ as regular
level sets with $\phi|_{\p_-W}=\min\,\phi$ and
$\phi|_{\p_+W}=\max\,\phi$ in the cobordism
case. The map $\fM:\Weinstein \to \Morse$ is the obvious one
$(\om,X,\phi)\to\phi$. 

The space $\Liouville$ of {\em Liouville structures} consists of
pairs $(\om,X)$ of a symplectic form $\om$ and a vector
field $X$ (the {\em Liouville field}) satisfying $L_X\om=\om$.  
Moreover, in the cobordism case we require that the Liouville field
$X$ points inward along $\p_-W$ and outward along
$\p_+W$, and in the manifold case we require that $X$ is complete and
there exists an exhaustion $V_1\subset V_2\subset\cdots$ of
$V=\cup_kV_k$ by compact sets with smooth boundary $\p V_k$ along
which $X$ points outward. The map $\Weinstein \to \Liouville$ sends
$(\om,X,\phi)$ to $(\om,X)$. Note that to each Liouville structure
$(\om,X)$ we can associate the {\em Liouville form} $\lambda:=i_X\om$,
and $(\om,X)$ can be recovered from $\lambda$ by the formulas
$\om=d\lambda$ and $i_Xd\lambda=\lambda$.  

The space $\Stein$ of {\em Stein structures} consists of pairs
$(J,\phi)$ of an integrable complex structure $J$ and a 
generalized Morse function $\phi$ (exhausting resp.~constant on the
boundary components) such that $-dd^\C\phi(v,Jv)>0$ for all $0\neq
v\in TV$, where $d^\C\phi:=d\phi\circ J$. 
If $(J,\phi)$ is a Stein structure, then $\om_\phi:=-dd^\C\phi$
is a symplectic form compatible with $J$. Moreover, the Liouville
field $X_\phi$ defined by
$$
   i_{X_\phi}\om_\phi = -d^\C\phi
$$ 
is the gradient of $\phi$ with respect to the Riemannian metric
$g_\phi:=\om_\phi(\cdot,J\cdot)$. In the manifold
case, completeness of $X_\phi$ can be arranged by replacing $\phi$ by
$f\circ\phi$ for a diffeomorphism $f:\R\to\R$ with $f''\geq 0$ and
$\lim_{x\to\infty}f'(x)=\infty$; we will suppress the function $f$
from the notation. So we have a canonical map
$$
   \fW: \Stein \to \Weinstein, \qquad
   (J,\phi)\mapsto(\om_\phi,X_\phi,\phi).
$$

It is interesting to compare the homotopy types of these spaces. For
simplicity, let us consider the case of a compact domain $W$ and equip
all spaces with the $C^\infty$ topology. The results which we discuss
below remain 
true in the manifold case, but one needs to define the topology more
carefully; see Section~\ref{ss:W-homotopies} below. 
Since all the spaces have the homotopy types of CW complexes, any weak
homotopy equivalence between them is a homotopy equivalence. 

The spaces $\Liouville$ and $\Weinstein$ are very
different: there exist many examples of Liouville domains that admit
no Weinstein structure, and of contact manifolds that bound a
Liouville domain but no Weinstein domain. The first such example was
constructed by McDuff~\cite{McD91}: the manifold $[0,1]\times\Sigma$,
where $\Sigma$ is the unit cotangent bundle of a closed oriented
surface of genus $>1$, carries a Liouville structure, but its boundary
is disconnected and hence cannot bound a Weinstein domain. Many more
such examples are discussed in~\cite{Gei94}. 

By contrast, the spaces of Stein and Weinstein
structures turn out to be closely related. One of the main results of
the book~\cite{CieEli12} is 

\begin{thm}
The map $\fW:\Stein\to\Weinstein$ induces an
isomorphism on $\pi_0$ and a surjection on $\pi_1$. 
\end{thm}

It lends evidence to the conjecture that $\fW:\Stein\to\Weinstein$ is
a homotopy equivalence.  

The relation between the spaces $\Morse$ and
$\Weinstein$ is the subject of this article. Note first that, since for
a Weinstein domain $(W,\om,X,\phi)$ of real dimension $2n$ all
critical points of $\phi$ have index $\leq n$, one should only
consider the subset $\Morse_n\subset\Morse$ of functions all of whose
critical points have index $\leq n$. Moreover, one should restrict to
the subset $\Weinstein_\eta^\flex\subset\Weinstein$ of 
Weinstein structures $(\om,X,\phi)$ with $\om$ in a fixed given
homotopy class $\eta$ of nondegenerate $2$-forms which are {\em
  flexible} in the sense of Section~\ref{sec:flex} below. The
following sections are devoted to the proof of

\begin{thm}[\cite{CieEli12}]\label{thm:combined}
Let $\eta$ be a nonempty homotopy class of nondegenerate $2$-forms on
a domain or manifold of dimension $2n>4$. Then: 

(a) Any Morse function $\phi\in\Morse_n$ can be lifted to a flexible
Weinstein structure $(\om,X,\phi)$ with $\om\in\eta$. 

(b) Given two flexible Weinstein structures
$(\om_0,X_0,\phi_0), (\om_1,X_1,\phi_1)\in\Weinstein_\eta^\flex$, any path
$\phi_t\in\Morse_n$, $t\in[0,1]$, connecting $\phi_0$ and $\phi_1$
can be lifted to a path of flexible Weinstein structures
$(\om_t,X_t,\phi_t)$ connecting $(\om_0,X_0,\phi_0)$ and
$(\om_1,X_1,\phi_1)$. 
\end{thm}

In other words, the map $\fM:\Weinstein_\eta^\flex\to\Morse_n$ has the
following properties: 
\begin{itemize}
\item $\fM$ is surjective;
\item the fibers of $\fM$ are path connected;
\item $\fM$ has the path lifting property. 
\end{itemize}

This motivates the following

\begin{conjecture}
On a domain or manifold of dimension $2n>4$, the map
$\fM:\Weinstein_\eta^\flex\to\Morse_n$ is a Serre fibration with
contractible fibers. 
\end{conjecture}

\section{Flexible Weinstein structures}\label{sec:flex}


Roughly speaking, a Weinstein structure is ``flexible'' if all its
attaching spheres obey an $h$-principle. More precisely,
note that each Weinstein manifold or cobordism can be cut along
regular level sets of the function into Weinstein cobordisms that are
elementary in the sense that there are no trajectories of the
vector field connecting different critical points. 
An elementary $2n$-dimensional Weinstein cobordism $(W,\om,X,\phi)$,
$n>2$, is called {\em flexible} if the attaching spheres of all
index $n$ handles form in $\p_-W$ a {\it loose} Legendrian link in the
sense of Section~\ref{sec:loose} below.  
A Weinstein cobordism or manifold structure $(\om,X,\phi)$ is called
flexible if it can be decomposed into elementary flexible
cobordisms. 

A $2n$-dimensional Weinstein structure $(\omega,X,\phi)$, $n\geq 2$,
is called {\em subcritical} if all critical points of the function
$\phi$ have index $<n$. In particular, any subcritical Weinstein
structure in dimension $2n>4$ is flexible. 

The notion of flexibility can be extended to dimension $4$ as
follows. We call a $4$-dimensional Weinstein cobordism {\em flexible}
if it is either subcritical, or the contact structure on $\p_-W$ is
overtwisted (or both); see Section~\ref{sec:ot} below. In particular,
a $4$-dimensional Weinstein {\em manifold} is then flexible if and
only if it is subcritical.  
  
\begin{remark}
The property of a Weinstein structure being subcritical is 
not preserved under Weinstein homotopies because one can always create
index $n$ critical points (see Proposition~\ref{prop:creation} below).  
We do not know whether flexibility is preserved under Weinstein
homotopies. In fact, it is not even clear to us whether every
decomposition of a flexible Weinstein cobordism $W$ into elementary
cobordisms consists of flexible elementary cobordisms. Indeed, 
if $\PP_1$ and $\PP_2$ are two partitions of $W$ into elementary
cobordisms and $\PP_2$ is finer than $\PP_1$, then flexibility of
$\PP_1$ implies flexibility of $\PP_2$ (in particular the partition
for which each elementary cobordism contains only one critical value
is then flexible), but we do not know whether flexibility of
$\PP_2$ implies flexibility of $\PP_1$. 
\end{remark}

The remainder of this section is devoted to the definition of loose
Legendrian links and a discussion of the relevant $h$-principles.

\subsection{Gromov's $h$-principle for subcritical isotropic
  embeddings}\label{sec:isotropic}

Consider a contact manifold $(M,\xi=\ker\alpha)$ of dimension $2n-1$ and a
manifold $\Lambda$ of dimension $k-1\leq n-1$. A {\em monomorphism}
$F:T\Lambda\to TM$ is a fiberwise injective bundle homomorphism
covering a smooth map $f:\Lambda\to M$. It is called {\em isotropic}
if it sends each $T_x\Lambda$ to a symplectically isotropic subspace
of $\xi_{f(x)}$ (with respect to the symplectic form $d\alpha|_\xi$). 
A {\em formal isotropic embedding}
of $\Lambda$ into $(M,\xi)$ is a
pair $(f,F^s)$, where $f:\Lambda\into M$ is a smooth embedding  
and $F^s:T\Lambda\to TM$, $s\in[0,1]$, is a homotopy of monomorphisms
covering $f$ that starts at $F^0=df$ and ends at an isotropic
monomorphism $F^1:T\Lambda\to \xi$. In the case $k=n$
we also call this a {\em formal Legendrian embedding}. 

Any genuine isotropic embedding can be viewed as a formal
isotropic embedding $(f,F^s\equiv df)$. We will not distinguish
between an isotropic embedding and its canonical lift to the space of
formal isotropic embeddings.  
A homotopy of formal isotropic embeddings $(f_t,F^s_t)$, $t\in[0,1]$, 
will be called a {\em formal isotropic isotopy}. Note that the maps
$f_t$ underlying a formal isotropic isotopy form a smooth isotopy. 

In the {\em subcritical} case $k<n$, Gromov proved 
the following $h$-principle. 
 
\begin{theorem}[$h$-principle for subcritical isotropic
  embeddings~\cite{Gro86,EliMis02}]~~
\label{thm:h-isotropic-subcrit} 

Let $(M,\xi)$ be a contact manifold of dimension $2n-1$ and $\Lambda$
a manifold of dimension $k-1<n-1$. Then the inclusion of the space
of isotropic embeddings $\Lambda\into (M,\xi)$ into the space of formal
isotropic embeddings is a weak homotopy equivalence. In particular: 

(a) Given any formal isotropic embedding $(f,F^s)$ of $\Lambda$ into
$(M,\xi)$, there exists an isotropic embedding $\wt f:\Lambda\into M$
which is $C^0$-close to $f$ and formally isotropically isotopic to
$(f,F^s)$.  

(b) Let $(f_t,F_t^s)$, $t\in[0,1]$, be a formal isotropic isotopy
connecting two isotropic embeddings $f_0,f_1:\Lambda\into M$. Then 
there exists an isotropic isotopy $\wt f_t$ connecting $\wt f_0=f_0$
and $\wt f_1=f_1$ which is $C^0$-close to $f_t$ and is homotopic to
the formal isotopy $(f_t,F_t^s)$ through formal isotropic isotopies
with fixed endpoints. 
\end{theorem} 

Let us discuss what happens with this theorem in the critical case
$k=n$. Part (a) remains true in all higher dimensions $k=n>2$:

\begin{theorem}[Existence theorem for Legendrian embeddings for
    $n>2$~\cite{Eli90, CieEli12}\footnote{ 
The hypothesis in~\cite{CieEli12} that $\Lambda$ is simply
connected can be easily removed.}]~~
\label{thm:h-Leg-emb}

Let $(M,\xi)$ be a contact manifold of dimension $2n-1\geq 5$ and
$\Lambda$ a manifold of dimension $n-1$. Then given any formal
Legendrian embedding $(f,F^s)$ of $\Lambda$ into 
$(M,\xi)$, there exists a Legendrian embedding $\wt f:\Lambda\into M$
which is $C^0$-close to $f$ and formally Legendrian isotopic to
$(f,F^s)$.  
\end{theorem}

Part (b) of Theorem~\ref{thm:h-isotropic-subcrit} does not carry over
to the critical case $k=n$: For any $n\geq 2$, there are many examples
of pairs of Legendrian knots in $(\R^{2n-1},\xi_\st)$ which are
formally Legendrian isotopic but not Legendrian isotopic; see
e.g.~\cite{Che02,EES05}.

\subsection{Legendrian knots in overtwisted contact manifolds}
\label{sec:ot}

Finally, let us consider Theorem~\ref{thm:h-isotropic-subcrit} in the
case $k=n=2$, i.e., for Legendrian knots (or links) in contact
$3$-manifolds. 
Recall that in dimension $3$ there is a dichotomy between tight
and overtwisted contact structures, which was introduced in
\cite{Eli89}. A contact structure $\xi$ on a $3$-dimensional manifold
$M$ is called {\em overtwisted} if there exists an embedded disc
$D\subset M$ which is tangent to $\xi$ along its boundary $\p D$. 
Equivalently, one can require the existence of an embedded disc with
Legendrian boundary $\p D$ which is transverse to $\xi$ along
$\p D$. A disc with such properties is called an
{\em overtwisted disc}. 

Part (a) of Theorem~\ref{thm:h-isotropic-subcrit} becomes false 
for $k=n=2$ due to Bennequin's inequality. Let
us explain this for $\R^3$ with its standard (tight) contact structure
$\xi_\st=\ker\alpha_\st$, $\alpha_\st=dz-p\,dq$. 
To any formal Legendrian embedding $(f,F^s)$ of $S^1$ into
$(\R^3,\xi_\st)$ we can associate two integers as follows. Identifying
$\xi_\st\cong\R^2$ via the projection $\R^3\to\R^2$ onto the
$(q,p)$-plane, the fiberwise injective bundle homomorphism $F^1:TS^1\cong
S^1\times\R\to\xi_\st\cong\R^2$ gives rise to a map $S^1\to
\R^2\setminus 0$, $t\mapsto F^1(t,1)$. The winding number of this map
around $0\in\R^2$ is called the {\em rotation number} $r(f,F^1)$. On
the other hand, $(F^1,iF^1,\p_z)$ defines a trivialization of the
bundle $f^*T\R^3$, where $i$ is the standard complex structure on
$\xi_\st\cong\R^2\cong\C$. Using the homotopy $F^s$, we homotope this
to a trivialization $(e_1,e_2,e_3)$ of $f^*T\R^3$ with $e_1=\dot f$
(unique up to homotopy). The {\em Thurston--Bennequin invariant}
$\tb(f,F^s)$ is the linking number of $f$ with a push-off in direction
$e_2$. It is not hard to see that the pair of invariants $(r,\tb)$
yields a bijection between homotopy classes of formal Legendrian
embeddings covering a fixed smooth embedding $f$ and $\Z^2$. In
particular, the pair $(r,\tb)$ can take arbitrary values on formal
Legendrian embeddings, while for genuine Legendrian embeddings
$f:S^1\into(\R^3,\xi_\st)$ the values of $(r,\tb)$ are constrained by
{\em Bennequin's inequality} (\cite{Ben83})  
$$
   \tb(f) + |r(f)| \leq -\chi(\Sigma),
$$
where $\Sigma$ is a Seifert surface for $f$. 

Bennequin's inequality, and thus the failure of part (a), carry over
to all tight contact $3$-manifolds. On the other hand, Bennequin's
inequality fails, and except for the $C^0$-closeness
Theorem~\ref{thm:h-isotropic-subcrit} remains true, on overtwisted
contact $3$-manifolds: 

\begin{thm}[\cite{Dym01,EliFra09}]
\label{thm:loose-3}
Let $(M,\xi)$ be a closed connected overtwisted contact $3$-manifold,
and $D\subset M$ an overtwisted disc. 

(a) Any formal Legendrian knot $(f,F^s)$ in $M$ 
is formally Legendrian isotopic 
to a Legendrian knot $\wt f:S^1\into M\setminus D$. 

(b) Let $(f_t,F^s_t)$, $s,t\in[0,1]$, be a formal Legendrian
isotopy in $M$ connecting two Legendrian knots $f_0,f_1:S^1\into
M\setminus D$. Then there exists a Legendrian isotopy $\wt
f_t:S^1\into M\setminus D$ connecting $\wt f_0=f_0$ 
and $\wt f_1=f_1$ which is homotopic to $(f_t,F^s_t)$ through formal
Legendrian isotopies with fixed endpoints. 
\end{thm}

Although Theorem~\ref{thm:h-isotropic-subcrit} (b) generally fails for
knots in tight contact $3$-manifolds, there are some remnants for
special classes of Legendrian knots: 
\begin{itemize}
\item any two formally Legendrian isotopic {\em unknots} in
  $(\R^3,\xi_\st)$ are Legendrian isotopic~\cite{EliFra09};
\item any two formally Legendrian isotopic knots become Legendrian
  isotopic after sufficiently many stabilizations (whose number
  depends on the knots)~\cite{FukTab97}. 
\end{itemize}
In~\cite{Mur11}, E.~Murphy discovered that the situation becomes much
cleaner for $n>2$: on any contact manifold of dimension $\geq 5$ there
exists a class of Legendrian knots, called {\it loose}, which 
satisfy both parts of Theorem~\ref{thm:h-isotropic-subcrit}. Let us
now describe this class.

\subsection{Murphy's $h$-principle for loose Legendrian
  knots}\label{sec:loose} 

In order to define loose Legendrian knots we need to describe a local
model. Throughout this section we assume $n>2$. 

Consider first a Legendrian arc $\lambda_0$ in the standard contact
space $(\R^3,dz-p_1dq_1)$ with front projection as shown in
Figure~\ref{fig:loose1}, for some $a>0$. 
\begin{figure}
\centering
\includegraphics{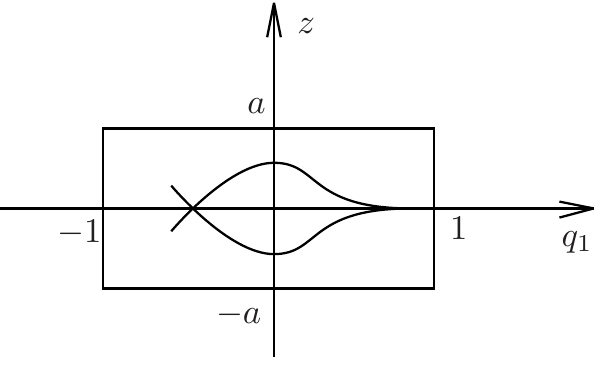}
\caption{Front of the Legendrian arc $\lambda_0$.}
\label{fig:loose1}
\end{figure}
Suppose that the slopes at the self-intersection point are $\pm 1$ and
the slope is everywhere in the interval 
$[-1,1]$, so the Legendrian arc $\lambda_0$ is contained in the box 
$$
   Q_a := \{|q_1|,|p_1|\leq 1,|z|\leq a\}
$$  
and $\p\lambda_0\subset\p Q_a$. 
Consider now the standard contact space $(\R^{2n-1},dz-\sum_{i=1}^{n-1}
p_idq_i)$, which we view as the product of the contact space $(\R^3,
dz-p_1dq_1)$ and the Liouville space $(\R^{2n-4},
-\sum\limits_{i=2}^{n-1} p_i dq_i)$. We set $q':=(q_2,\dots, q_{n-1}\}$ and
$p':=(p_2,\dots, p_{n-1})$. For $b,c>0$ we define
\begin{align*}
   P_{bc} &:= \{|q'|\leq b,\;|p'|\leq c\}\subset\R^{2n-4}, \cr
   R_{abc} &:= Q_a \times P_{bc} = \{|q_1|,|p_1|\leq  1,\; |z|\leq a,\;|q'|\leq
   b,\; |p'|\leq c\}.
\end{align*}
Let the Legendrian solid cylinder 
$\Lambda_0 \subset (\R^{2n-1},dz-\sum_{i=1}^{n-1} p_idq_i)$ be
the product of $\lambda_0\subset\R^3$ with the Lagrangian disc
$\{p'=0,|q'|\leq b\}\subset \R^{2n-4}$. Note that $\Lambda_0\subset
R_{abc}$ and $\p\Lambda_0\subset \p  R_{abc}$. 
The front of $\Lambda_0$ is obtained by translating the front of
$\lambda_0$ in the $q'$-directions; see Figure~\ref{fig:loose2}. 
\begin{figure}
\centering
\includegraphics{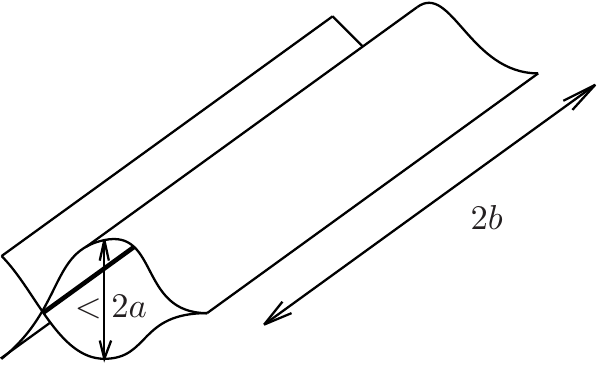}
\caption{Front of the Legendrian solid cylinder $\Lambda_0$.}
\label{fig:loose2}
\end{figure}
The pair $(R_{abc},\Lambda_0)$ is called a {\it standard loose 
  Legendrian chart} if 
$$
   a<bc.
$$ 
Given any contact manifold $(M^{2n-1},\xi)$, a Legendrian submanifold
$\Lambda\subset M$ with connected components $\Lambda_1,\dots, 
\Lambda_k$ is called 
{\em loose} 
\index{loose Legendrian submanifold}
if there exist  Darboux charts
$U_1,\dots, U_k\subset M$ such that $\Lambda_i\cap U_j=\varnothing$ for $i\neq j$ 
and each pair $(U_i,\Lambda_i\cap
U_i)$, $i=1,\dots, k$, is isomorphic to 
a standard loose Legendrian chart $(R_{abc},\Lambda_0)$. 
A Legendrian embedding $f:\Lambda\into M$ is called loose if its
image is a loose Legendrian submanifold. 
 
\begin{remark}\label{rem:loose-knots}
(1) By the contact isotopy extension theorem,
looseness is 
preserved under Legendrian isotopies within a fixed contact manifold. 
Since the model $\Lambda_0$ above can be slightly extended to a
Legendrian disc in standard $\R^{2n-1}$, 
and any two Legendrian discs are isotopic, 
it follows that {\it any Legendrian disc is loose}. 

(2) By rescaling $q'$ and $p'$ with inverse factors one can always
achieve $c=1$ in the definition of a standard loose Legendrian chart. However,
the inequality $a<bc$ is absolutely crucial in the definition. Indeed,
it follows from Gromov's isocontact embedding theorem
that around {\it any}
point in {\it any} Legendrian submanifold $\Lambda$ one can find a
Darboux neighborhood $U$ such that the pair $(U,\Lambda\cap U)$
is isomorphic to $(R_{1b1},\Lambda_0)$ for some sufficiently small
$b>0$. 

(3) Figure~\ref{fig:squeeze}
\begin{figure}
\centering
\includegraphics[height=50mm, width=80mm]{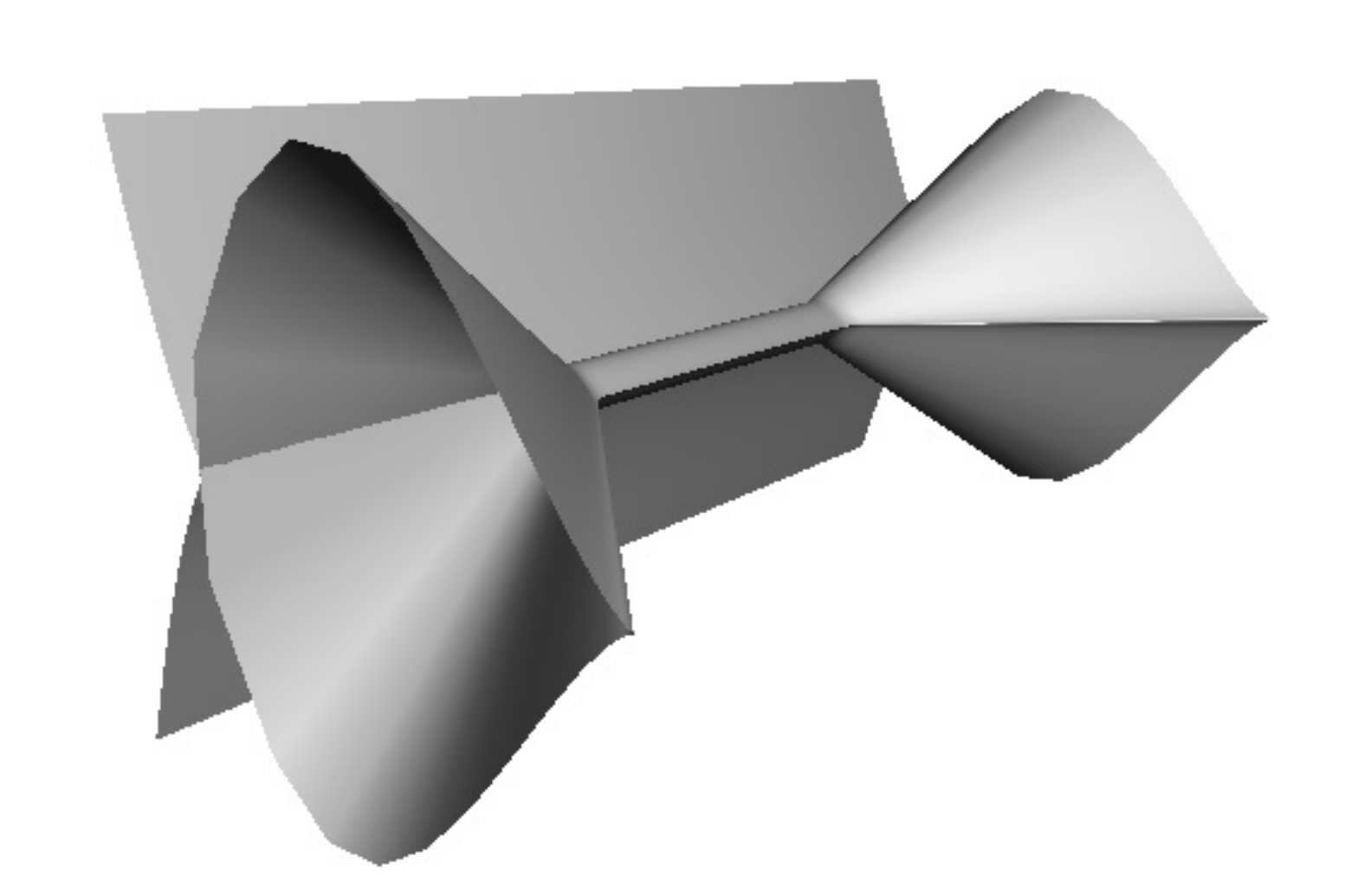}
\caption{Shrinking a standard loose Legendrian chart (picture is courtesy of E. Murphy).}
\label{fig:squeeze}
\end{figure}
taken from~\cite{Mur11} shows that the definition of looseness does
not depend on the precise choice of the standard loose Legendrian chart
$(R_{abc},\Lambda_0)$: Given a standard loose Legendrian chart with
$c=1$, the condition $a<b$ allows us to shrink its front in
the $q'$-directions, keeping it fixed near the boundary and with all
partial derivatives in $[-1,1]$ (so the deformation remains in the
Darboux chart $R_{ab1}$), to another standard loose Legendrian
chart $(R_{a'b'1},\Lambda_0')$ with $b'\geq(b-a)/2$ and arbitrarily
small $a'>0$. Moreover, we can arbitrarily prescribe the shape of the
cross section $\lambda_0'$ of $\Lambda_0'$ in this process.  
So if a Legendrian submanifold is loose for some
model $(R_{abc},\Lambda_0)$, then it is also loose for any other
model. In particular, 
fixing $b,c$ we can make $a$ arbitrarily small, and we can create
arbitrarily many disjoint standard loose Legendrian charts. 
\end{remark}

Now we can state the main result from~\cite{Mur11}. 

\begin{thm}[Murphy's $h$-principle for loose embeddings~\cite{Mur11}]~~
\label{thm:loose}

Let $(M,\xi)$ be a contact manifold of dimension $2n-1\geq 5$ and
$\Lambda$ a manifold of dimension $n-1$. Then:   

(a) Given any formal Legendrian embedding $(f,F^s)$ of $\Lambda$ into
$(M,\xi)$, there exists a loose Legendrian embedding $\wt
f:\Lambda\into M$ which is $C^0$-close to $f$ and formally Legendrian
isotopic to $(f,F^s)$.  

(b) Let $(f_t,F_t^s)$, $t\in[0,1]$, be a formal Legendrian isotopy
connecting two loose Legendrian embeddings $f_0:f_1:\Lambda\into
M$. Then there exists a Legendrian isotopy $\wt f_t$ connecting $\wt
f_0=f_0$ and  $\wt f_1=f_1$ which is $C^0$-close to $f_t$ and is
homotopic to the formal isotopy $(f_t,F_t^s)$ through formal
Legendrian isotopies with fixed endpoints. 
\end{thm}

Part (a) of this theorem is a consequence of
Theorem~\ref{thm:h-Leg-emb} and the stabilization construction
which we describe next.

\subsection{Stabilization of Legendrian submanifolds}

Consider a Legendrian submanifold $\Lambda_0$ in a contact manifold
$(M,\xi)$ of dimension $2n-1$. Near a point of $\Lambda_0$, pick
Darboux coordinates $(q_1,p_1,\dots,q_{n-1},p_{n-1},z)$ in which
$\xi=\ker(dz-\sum_jp_jdq_j)$ and the front projection of $\Lambda_0$
is a standard cusp $z^2=q_1^3$. Deform the two branches of the front
to make them parallel over some open ball
$B^{n-1}\subset\R^{n-1}$. After rescaling, we may thus assume that the
front of $\Lambda_0$ has two parallel branches $\{z=0\}$ and $\{z=1\}$
over $B^{n-1}$, see Figure~\ref{fig:stab}. 
\begin{figure}
\centering
\includegraphics{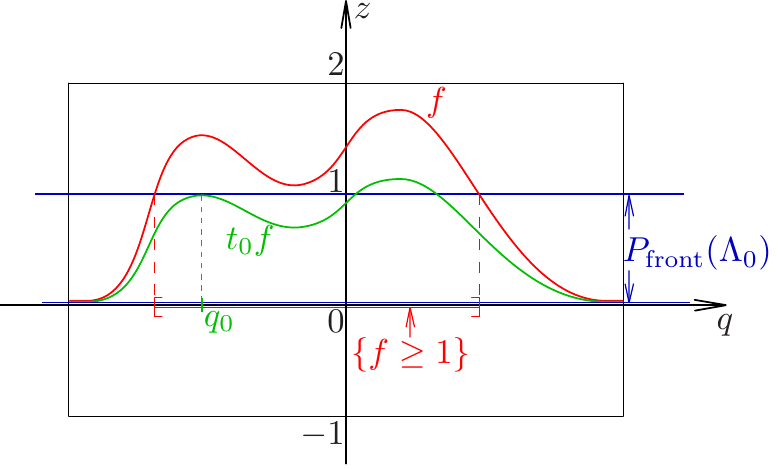}
\caption{Stabilization of a Legendrian submanifold.}
\label{fig:stab}
\end{figure}

Pick a non-negative function $\phi:B^{n-1}\to\R$ with compact support and
$1$ as a regular value, so $N:=\{\phi\geq 1\}\subset B^{n-1}$ is a
compact manifold with boundary. Replacing for each $t\in[0,1]$ the
lower branch $\{z=0\}$ by the graph $\{z=t\phi(q)\}$ of the function $t\phi$
yields the fronts of a path of Legendrian immersions $\Lambda_t\subset
M$ connecting $\Lambda_0$ to a new Legendrian submanifold $\Lambda_1$. 
Note that $\Lambda_t$ has a self-intersection for each critical point
of $t\phi$ on level $1$. 

We count the self-intersections with signs as follows. Consider the
immersion $\Gamma:=\bigcup_{t\in[0,1]}\Lambda_t\times\{t\}\subset
M\times[0,1]$. After a generic perturbation, we may assume that
$\Gamma$ has finitely many transverse self-intersections and define
its {\em self-intersection index}
$$
   I_\Gamma := \sum_p I_\Gamma(p) \in 
   \begin{cases}
      \Z & \text{ if $n$ is even,} \\
      \Z_2 & \text{ if $n$ is odd}
   \end{cases}
$$
as the sum over the indices of all self-intersection points $p$. Here
the index $I_\Gamma(p)=\pm 1$ is defined by comparing the orientations
of the two intersecting branches of $\Gamma$ to the orientation of
$M\times[0,1]$. For $n$ even this does not depend on the order of
the branches and thus gives a well-defined integer, while for $n$ odd
it is only well-defined mod $2$. By a theorem of Whitney~\cite{Whi44},
for $n>2$, the regular homotopy $\Lambda_t$ can be deformed through
regular homotopies fixed at $t=0,1$ to an isotopy iff $I_\Gamma=0$. 

\begin{proposition}[\cite{Mur11}]\label{prop:stab}
For $n>2$, the Legendrian regular homotopy $\Lambda_t$ obtained
from the stabilization construction over a nonempty domain $N\subset
B^{n-1}$ has the following properties. 

(a) $\Lambda_1$ is {\em loose}. 

(b) If $\chi(N)=0$, then $\Lambda_1$ is formally Legendrian isotopic
to $\Lambda_0$.  

(c) The regular homotopy $(\Lambda_t)_{t\in[0,1]}$ has
self-intersection index $(-1)^{(n-1)(n-2)/2}\chi(N)$.  
\end{proposition}

\begin{proof}
(a) Recall that in the stabilization construction we choose a Darboux
chart in which the front of $\Lambda_0$ consists of the two branches
$\{z=\pm q_1^{3/2}\}$ of a standard cusp, and then deform the lower branch
to the graph of a function $\phi$ which is bigger than $q_1^{3/2}$ over a
domain $N\subset\R^{n-1}$; see Figure~\ref{fig:stab-loose}.  
\begin{figure}
\centering
\includegraphics[height=60mm]{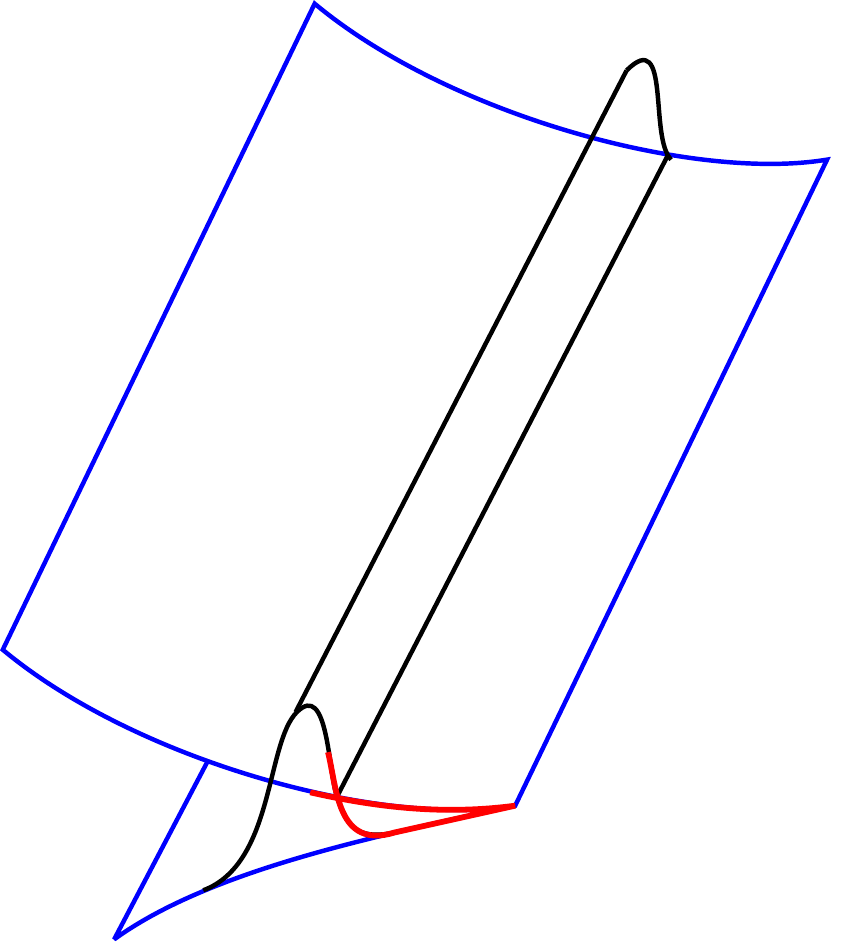}
\caption{A standard loose Legendrian chart appears in the
  stabilization procedure.}
\label{fig:stab-loose}
\end{figure}
Performing this construction sufficiently close to the cusp edge,
we can keep the values and the differential of the function
$\phi$ arbitrarily small. Then the deformation is localized within the
chosen Darboux neighborhood, and comparing
Figures~\ref{fig:stab-loose} and~\ref{fig:loose2} we see that
$\Lambda_1$ is loose.

(b) 
Consider again the stabilization construction on the two
parallel branches $\{z=0\}$ and $\{z=1\}$ of $\Lambda_0$ over the domain
$N=\{\phi\geq 1\}$. 
Since $\chi(N)=0$, there exists a nowhere vanishing vector field $v$ 
on $N$ which agrees with $\nabla\phi$ near $\p N$. Linearly
interpolating the $p$-coordinate of $\Lambda_1$ from $\nabla\phi(q)$ to
$v(q)$ (keeping $(q,z)$ fixed), then pushing the $z$-coordinate down
to $0$ (keeping $(q,p)$ fixed), and finally linearly interpolating
$v(q)$ to $0$ (keeping $(q,z)$ fixed) defines a smooth isotopy
$f_t:\Lambda_0\into M$ from $f_0=\id:\Lambda_0\to\Lambda_0$ to a
parametrization $f_1:\Lambda_0\to\Lambda_1$. On the other 
hand, the graphs of the functions $t\phi$ define a Legendrian regular
homotopy from $f_0$ to $f_1$, so their differentials give a path of
Legendrian monomorphisms $F_t$ from $F_0=df_0$ to $F_1=df_1$. Now note
that over the region $N$ all the $df_t$ and $F_t$ project as the
identity onto the $q$-plane, so linearly connecting $df_t$ and $F_t$
yields a path of monomorphisms $F^s_t$, $s\in[0,1]$, and hence the
desired formal Legendrian isotopy $(f_t,F_t^s)$ from $f_0$ to $f_1$. 

To prove (c), make the function $\phi$ Morse on $N$ and apply the
Poincar\'e--Hopf index theorem. 
\end{proof}

Since for $n>2$ there exist domains $N\subset\R^{n-1}$ of 
arbitrary Euler characteristic $\chi(N)\in\Z$, we can apply
Proposition~\ref{prop:stab} in two ways: 
Choosing $\chi(N)=0$, we can $C^0$-approximate every Legendrian
submanifold $\Lambda_0$ by a loose one which is formally Legendrian
homotopic to $\Lambda_0$. Combined with Theorem~\ref{thm:h-Leg-emb},
this proves Theorem~\ref{thm:loose}(a). 
Choosing $\chi(N)\neq 0$, we can connect each Legendrian
submanifold $\Lambda_0$ to a (loose) Legendrian submanifold
$\Lambda_1$ by a Legendrian regular homotopy $\Lambda_t$ with any
prescribed self-intersection index. This will be a crucial ingredient
in the proof of existence of Weinstein structures. 

\begin{remark}
For $n=2$ we can still perform the stabilization
construction. However, since every domain $N\subset\R$ is a union of
intervals, the self-intersection index $\chi(N)$ in
Proposition~\ref{prop:stab} is now always {\em positive} and hence
cannot be arbitrarily prescribed. Figure~\ref{fig:stab-3} 
\begin{figure}
\centering
\includegraphics{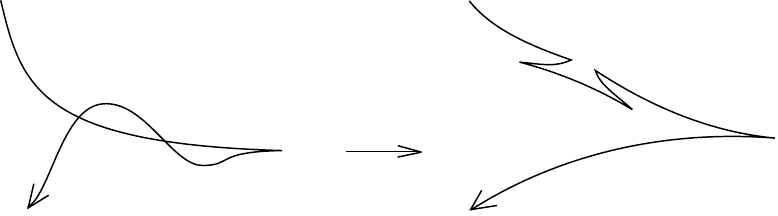}
\caption{Stabilization in dimension $3$.}
\label{fig:stab-3}
\end{figure}
shows two
front projections of the stabilization over an interval (related by
Legendrian Reidemeister moves, see e.g.~\cite{Etn05}). Thus our
stabilization introduces a downward and an upward zigzag, which
corresponds to a ``positive and a negative stabilization'' in the
usual terminology of $3$-dimensional contact topology. It leaves the
rotation number unchanged and {\em decreases} the Thurston--Bennequin
invariant by $2$, in accordance with Bennequin's inequality.
In particular, stabilization in dimension $3$ never preserves the
formal Legendrian isotopy class.
\end{remark}

\subsection{Totally real discs attached to a contact boundary}
\label{sec:attaching-discs}

The following Theorems~\ref{thm:real-handle}
and~\ref{thm:real-handle-subcrit} are combinations of the  
$h$-principles discussed in Sections~\ref{sec:isotropic}
and~\ref{sec:loose} with Gromov's $h$-principle for totally real
embeddings~\cite{Gro86}.

Let $(V, J)$ be an almost complex manifold and $W\subset V$ a 
domain with smooth boundary $\p W$.  Let $L$ be a 
compact manifold with boundary. Let $f:L\into 
V\setminus\Int W$  be an embedding with 
$f(\p L)=f(L)\cap\p W$
which is transverse to $\p W$ along $\p 
L$. We say in this case that $f$ {\em transversely} attaches $L$ to $W$
along $\p L$.  If, in addition, $Jdf(TL|_{\p L})\subset T(\p
W)$, then we say that $f$ attaches $L$ to $W$ {\em $J$-orthogonally}. 
Note that this implies that $df(\p L)$ is tangent to the {\em maximal
  $J$-invariant distribution} $\xi=T(\p W)\cap JT(\p W)$. In 
particular, if the distribution $\xi$ is a contact structure, 
then $f(\p L)$ is an isotropic
submanifold for the contact structure $\xi$. 

\begin{thm}\label{thm:real-handle}
Suppose that $(V,J)$ is an almost complex manifold of real dimension $2n$, 
and $W\subset V$ is a domain such that the distribution $\xi=T(\p
W)\cap JT(\p W)$ is contact. Suppose that an embedding  $f:D^k \into
V$, $k\leq n$, transversely attaches $D^k$ to $W$ along $\p D^k$. 
If $k=n=2$ we assume, in addition, that the induced contact structure on
$\p W$ is overtwisted.
Then there exists an isotopy $f_t:D^k \into V,$ $t\in[0,1]$, through
embeddings transversely attaching $D^k$ to $W$, such that $f_0=f$, and
$f_1$ is totally real and $J$-orthogonally attached to $W$. 
Moreover, in the case $k=n>2$ we can arrange that the Legendrian
embedding $f_1|_{\p D^k}:\p D^k\into\p W$ is loose, while   
for $k=n=2$ we can arrange that the complement $\p W\setminus f_t(\p
D^2)$ is overtwisted for all $t\in[0,1]$.  
\end{thm} 
 
We will  also need the following 1-parametric
version of Theorem~\ref{thm:real-handle} for totally real discs
attached along loose knots.  

\begin{thm}\label{thm:real-handle-subcrit}
Let $J_t$, $t\in[0,1]$, be a family of almost complex structures on a
$2n$-dimensional manifold $V$. Let $W\subset V$ be a domain with
smooth boundary such that the distribution $\xi_t=T(\p W)\cap J_tT(\p
W)$ is contact 
for each $t\in[0,1]$. Let $f_t:D^k\into V\setminus\Int W$,
$t\in[0,1]$, $k\leq n$, be an isotopy of embeddings transversely
attaching $D^k$ to $W$ along $\p D^k$. Suppose that for $i=0,1$ the
embedding $f_i$ is $J_i$-totally real and $J_i$-orthogonally attached
to $W$. Suppose that either $k<n$, or $k=n>2$ and the 
Legendrian embeddings $f_i|_{\p D}$, $i=0,1$ are {\em loose}.   
Then there exists a $2$-parameter family of embeddings
$f_t^s:D^k\into V\setminus\Int W$ with the following
properties:
\begin{itemize}
\item $f_t^s$ is transversely attached to $W$ along $\p D^k$ and
  $C^0$-close to $f_t$ for all $t,s\in[0,1]$;  
\item $f_t^0=f_t$ for all $t\in[0,1]$ and $f_0^s=f_0$, $f_1^s=f_1$ for
  all $s\in[0,1]$; 
\item $f_t^1$ is $J_t$-totally real and $J_t$-orthogonally attached to
$W$ along $\p D^k$ for all $t\in[0,1]$.
\end{itemize}
\end{thm}

\section{Morse preliminaries}\label{ch:Morse}

In this section we gather some notions and results from Morse theory
that are needed for our main results. We omit most of the proofs and
refer the reader to the corresponding chapter of our
book~\cite{CieEli12}. Throughout this section, $V$ denotes a smooth
manifold and $W$ a cobordism, both of dimension $m$. 

\subsection{Gradient-like vector fields}\label{sec:gradient}

A smooth function $\phi:V\to\R$ is called 
{\em Lyapunov } for a vector field
for $X$, and $X$  is called
{\em gradient-like} 
for $\phi$, if
\begin{equation}\label{eq:grad-like}
   X\cdot\phi\geq\delta(|X|^2+|d\phi|^2)
\end{equation}
for a positive function $\delta:V\to\R_+$, where $|X|$ is the norm
with respect to some Riemannian metric on $V$ and $|d\phi|$ is the
dual norm.  
By the Cauchy--Schwarz inequality, condition~\eqref{eq:grad-like}
implies 
\begin{equation}
   \delta|X|\leq |d\phi|\leq\frac{1}{\delta}|X|. 
\end{equation} 
In particular, 
zeroes of $X$ coincide with critical points of $\phi$.

\begin{lemma}\label{lem:cone}
(a) If $X_0,X_1$ are gradient-like vector fields for $\phi$,
then so is $f_0X_0+f_1X_1$ for any nonnegative functions $f_0,f_1$
with $f_0+f_1>0$. 

(b) If $\phi_0,\phi_1$ are Lyapunov functions for $X$,
then so is $\lambda_0\phi_0+\lambda_1\phi_1$ for any nonnegative
constants $\lambda_0,\lambda_1$ with $\lambda_0+\lambda_1>0$. 

In particular, the following spaces are convex cones and hence
contractible:
\begin{itemize}
\item the space of Lyapunov functions for a given vector
field $X$;
\item the space of gradient-like vector fields for a given
  function $\phi$.
\end{itemize}
\end{lemma}

\begin{proof}
Consider two vector fields $X_0,X_1$ satisfying
$X_i\cdot\phi\geq\delta_i(|X_i|^2+|d\phi|^2)$ and nonnegative
functions $f_0,f_1$ with $f_0+f_1>0$. 
Then the vector field
$X=f_0X_0+f_1X_1$ satisfies~\eqref{eq:grad-like} with
$\delta:=\min\left\{\frac{\delta_0}{2f_0},
\frac{\delta_1}{2f_1},f_0\delta_0+f_1\delta_1\right\}$: 
\begin{align*}
   X\cdot\phi 
   &\geq f_0\delta_0|X_0|^2+f_1\delta_1|X_1|^2 +
   (f_0\delta_0+f_1\delta_1)|d\phi|^2 \cr
   &\geq 2\delta (|f_0X_0|^2+|f_1X_1|^2) + \delta|d\phi|^2
   \cr
    &\geq \delta(|X|^2+|d\phi|^2). 
\end{align*}
Positive combinations of functions are treated analogously. 
\end{proof}

\subsection{Morse and Smale cobordisms}\label{sec:Smale-cobordisms}

A {\em (generalized) Morse cobordism} is a pair $(W,\phi)$, where $W$ is a
cobordism and $\phi:W\to\R$ is a (generalized) Morse function 
which has $\p_\pm W$ as its regular level sets such that
$\phi|_{\p_-W}<\phi|_{\p_+W}$. 
A {\em (generalized) Smale cobordism} is a triple $(W,\phi,X)$, where $(W,\phi)$ is
a (generalized) Morse cobordism and $X$ is a gradient-like vector field for
$\phi$. Note that $X$ points inward along $\p_-W$ and outward along
$\p_+W$. 
A generalized Smale cobordism $(W,\phi,X)$ is called {\em elementary} 
if there are no $X$-trajectories between
different critical points of $\phi$. 

If $(W,\phi,X)$ is an elementary generalized Smale cobordism, then the
stable manifold of each nondegenerate critical point $p$ is a 
disc $D_p^-$ which intersects $\p_-W$ along a sphere $S_p^-=\p D_p^-$.
We call $D_p^-$ and $S_p^-$ the {\em stable disc (resp.~sphere)} of $p$. 
Similarly, the unstable manifolds and their intersections with $\p_+W$
are called {\em unstable discs and spheres}. 
For an embryonic critical point $p$, the closure of the (un)stable
manifold is the {\em (un)stable half-disc $\wh D_p^\pm$} intersecting
$\p_\pm W$ along the hemisphere $\wh S_p^\pm$.  

An {\em admissible partition} of a generalized Smale cobordism 
$(W,\phi,X)$ is a finite sequence
$m=c_0<c_1<\dots<c_N=M$  of regular values of $\phi$, where we  
denote $\phi|_{\p_-W}=m$ and $\phi|_{ \p_+ W}=M$, such that each
subcobordism $W_k=\{c_{k-1}\leq \phi\leq c_k\}$, $k=1,\dots, N$, is
elementary. 
The following lemma is straightforward.

\begin{lemma}\label{lm:existence-partition-M}
Any generalized Smale cobordism admits an admissible partition into elementary
cobordisms. Similarly, for any exhausting generalized Morse function
$\phi$ and gradient-like vector field $X$ on a 
noncompact manifold $V$, one can find regular values
$c_0<\min\phi<c_1<\dots\to\infty$ such that the cobordisms
$W_k=\{c_{k-1}\leq\phi\leq c_k\}$, $k=1,\dots,$ are elementary. If
$\phi$ has finitely many critical points, then all but finitely many
of these cobordisms have no critical points. 
\end{lemma} 

\subsection{Morse and Smale homotopies}\label{sec:Smale-hom}

  
A smooth family $(W,\phi_t)$, $t\in[0,1]$, of generalized Morse
cobordism structures is called a {\em Morse homotopy} if there is a
finite set $A\subset(0,1)$ with the following properties: 
\begin{itemize}
\item for each $t\in A$ the function $\phi_t$ has a unique birth-death
  type critical point $e_t$ such that $\phi_t(e_t)\neq\phi_t(q)$ for
  all other critical points $q$ of $\phi_t$;  
\item for each $t\notin A$ the function $\phi_t$ is Morse.  
\end{itemize}
A {\em Smale homotopy} is a smooth family $(W,X_t,\phi_t)$,
$t\in[0,1]$, of generalized Smale cobordism structures such that
$(W,\phi_t)$ a Morse homotopy.
A Smale homotopy $\fS_t=(W,X_t,\phi_t)$, $t\in[0,1]$ is called an
{\em elementary Smale homotopy}  
of Type I, IIb, IId, respectively, if the following holds:
\begin{itemize}
\item {Type I.} $\fS_t$ is an elementary Smale cobordism for all
  $t\in[0,1]$. 
\item {Type IIb (birth).} There is $t_0\in(0,1)$ such that for $t<t_0$
  the  function $\phi_t$ has no critical points, $\phi_{t_0}$ has a
  birth type critical point, and for $t>t_0$ the
  function $\phi_t$  has has two  critical points $p_t$ and  $q_t$ of
  index  $i$ and $i-1$, respectively, connected by a unique
  $X_t$-trajectory. 
\item {Type IId (death).}  There is $t_0\in(0,1)$ such that for
  $t>t_0$ the function $\phi_t$  has no critical points,
  $\phi_{t_0}$ has a death type critical point, and for
  $t<t_0$ the function $\phi_t$ has two critical points $p_t$
  and  $q_t$ of index  $i$ and $i-1$, respectively, connected by a
  unique $X_t$-trajectory.  
\end{itemize}
We will also refer to an elementary Smale homotopy of Type IIb
(resp.~IId) as a 
{\em creation (resp.~cancellation) family}.  
  
An {\em admissible partition} of a Smale homotopy
$\fS_t=(W,X_t,\phi_t)$, $t\in[0,1]$, is 
a sequence $0=t_0<t_1<\dots<t_p=1$ of parameter
values, and for each $k=1,\dots, p$ a finite sequence of functions  
$$
   m(t)=c^k_0(t)<c^k_1(t)<\dots<c^k_{N_k}(t)=M(t), \qquad
   t\in[t_{k-1},t_k], 
$$
where $m(t):=\phi_t(\p_-W)$ and $M(t):=\phi_t(\p_+W)$, such that
$c^k_j(t)$, $j=0,\dots, N_k$ are regular values of $\phi_t$
and each Smale homotopy
$$
   \fS^k_j := \left(W^k_j(t):=\{c^k_{j-1}(t)\leq
   \phi_t\leq c^k_j(t)\}, X_t|_{W^k_j(t)},\phi_t|_{W^k_j(t)}\right)_{t\in[t_{k-1},t_k]} 
$$
is elementary.

\begin{lemma}\label{lm:admiss-homotopy}
Any Smale homotopy admits an admissible partition.  
\end{lemma}
 
\begin{proof}
Let $A\subset(0,1)$ be the finite subset in the definition of a Smale
homotopy. Using Lemma~\ref{lm:existence-partition-M}, we now first
construct an admissible partition on $\Op A$ and then extend it over
$[0,1]\setminus\Op A$. 
\end{proof}

\subsection{Equivalence of elementary Smale homotopies}

We define the {\em profile} (or Cerf diagram) of a Smale homotopy
$\fS_t=(W,X_t,\phi_t)$, $t\in[0,1]$, as the subset
$C(\{\phi_t\})\subset \R\times\R$ such that 
$C(\{\phi_t\})\cap (t\times\R)$ is the set of critical values of the
function $\phi_t$.  
We will use the notion of profile only for elementary homotopies.
   
The following two easy lemmas are proved in~\cite{CieEli12}. The first
one shows that if two elementary Smale homotopies have
the same profile, then their functions are related by
diffeomorphisms. 

\begin{lemma}\label{lm:Morse-profiles}
Let $\fS_t=(W,X_t,\phi_t)$ and $\wt\fS_t=(W,\wt X_t,\wt\phi_t)$,
$t\in[0,1]$, be two elementary Smale 
homotopies with the same profile such that $\fS_0=\wt\fS_0$.
Then there exists a diffeotopy $h_t:W\to W$ with $h_0=\id$ such that
$\phi_t=\wt\phi_t\circ h_t$ for all $t\in[0,1]$.  
Moreover, if $\phi_t=\wt\phi_t$ near $\p_+W$ and/or $\p_-W$ we can
arrange $h_t=\id$ near $\p_+W$ and/or $\p_-W$. 
\end{lemma}
 
The second lemma provides elementary Smale homotopies with prescribed
profile. 

\begin{lemma}\label{lem:prescribe-profile}
Let $(W,X,\phi)$ be an elementary Smale cobordism with $\phi|_{\p_\pm
  W}=a_\pm$ and critical points $p_1,\dots,p_n$ of values
$\phi(p_i)=c_i$. For $i=1,\dots,n$ let $c_i:[0,1]\to(a_-,a_+)$ be
smooth functions with $c_i(0)=c_i$. Then there exists a smooth family
$\phi_t$, $t\in[0,1]$, of Lyapunov functions for $X$ with
$\phi_0=\phi$ and $\phi_t=\phi$ on $\Op\p W$ such that
$\phi_t(p_i)=c_i(t)$. 
\end{lemma}

\subsection{Holonomy of Smale cobordisms}

Let $(W,X,\phi)$ be a Smale cobordism such that the function $\phi$
has no critical points. The 
{\em holonomy}
\index{holonomy} 
of $X$ is the diffeomorphism 
$$
   \Gamma_X:\p_+W\to\p_-W,
$$
which maps $x\in\p_+W$ to the intersection of its trajectory under the
flow of $-X$ with $\p_-W$. 

Consider now a Morse cobordism $(W,\phi)$ without critical points. 
Denote by $\XX(W,\phi)$ the space of all
gradient-like vector fields for $\phi$. Note that the holonomies of
all $X\in\XX(W,\phi)$ are isotopic. We denote by $\DD(\p_+W,\p_-W)$
the corresponding path component in the space of diffeomorphisms from
$\p_+W$ to $\p_-W$. All spaces are equipped with the
$C^\infty$-topology. 

Recall that a continuous map $p:E\to B$ is a {\em Serre fibration} 
if it has the homotopy lifting property for all closed discs $D^k$, i.e.,
given a homotopy $h_t:D^k\to B$, $t\in[0,1]$, and a lift $\wt
h_0:D^k\to E$ with $p\circ\wt h_0=h_0$, there exists a homotopy $\wt
h_t:D^k\to E$ with $p\circ\wt h_t=h_t$.  
We omit the proof of the following easy lemma. 

\begin{lemma}\label{lm:Morse-Serre}
Let $(W,\phi)$ be a Morse cobordism without critical points. Then
the map $\XX(W,\phi)\to \DD(\p_+W,\p_-W)$ that assigns to $X$ its
holonomy $\Gamma_X$ is a Serre fibration. In particular: 

(i) Given $X\in \XX(W,\phi)$ and an isotopy $h_t\in
  \DD(\p_+W,\p_-W)$, $t\in[0,1]$, with  $h_0=\Gamma_X$, there exists a
  path $X_t\in\XX(W,\phi)$ with $X_0=X$ such that $\Gamma_{X_t}=h_t$
  for all $t\in[0,1]$.
   
(ii) Given a path $X_t\in\XX(W,\phi)$, $t\in[0,1]$, and a path
  $h_t\in\DD(\p_+W,\p_-W)$ which is homotopic to $\Gamma_{X_t}$ with
  fixed endpoints, there exists a path $\wt X_t\in\XX(W,\phi)$ with
  $\wt X_0=X_0$ and $\wt X_1= X_1$ such that $\Gamma_{\wt X_t}= h_t$
  for all $t\in[0,1]$.
\end{lemma}

As a consequence, we obtain

\begin{lemma}\label{lm:fixing-A-Morse}
Let $X_t$, $Y_t$ be two paths in $\XX(W,\phi)$ starting at the
same point $X_0=Y_0$. Suppose that for a subset $A\subset \p_+W$,
one has $\Gamma_{X_t}(A)=\Gamma_{Y_t}(A)$ for all $t\in[0,1]$. 
Then there exists a path $\wh X_t\in\XX(W,\phi)$ such that
\begin{enumerate}
\item $\wh X_t=X_{2t}$ for $t\in[0,\frac12]$;
\item $\wh X_1=Y_1$;
\item $\Gamma_{\wh X_t}(A)=\Gamma_{Y_1}(A)$ for $t\in[\frac12,1]$.
\end{enumerate}
\end{lemma}

\begin{proof}
Consider the path $\gamma:[0,1]\to\DD(\p_+W,\p_-W)$ given by the
formula 
$$
  \gamma(t) := \Gamma_{X_1}\circ\Gamma_{X_t}^{-1}\circ\Gamma_{Y_t}.
$$
We have $\gamma(0)=\Gamma_{X_1}$ and $\gamma(1)=\Gamma_{Y_1}$.
The path $\gamma$ is homotopic with fixed endpoints to the
concatenation of the paths $\Gamma_{X_{1-t}}$  and
$\Gamma_{Y_t}$. Hence by Lemma~\ref{lm:Morse-Serre} we
conclude that there exists a path $X'_t\in\XX(W,\phi)$ such that  
$X_0'= X_1$, $X'_1=Y_1$, and $\Gamma_{X'_t}=\gamma(t)$ for all
$t\in[0,1]$. Since 
$$
   \Gamma_{X'_t}(A) = \Gamma_{X_1}\left(\Gamma_{X_t}^{-1}
     \bigl(\Gamma_{Y_t}(A)\bigr)\right) 
   = \Gamma_{X_1}(A) = \Gamma_{Y_1}(A), 
$$
the concatenation $\wh X_t$ of the paths $X_t$ and $X_t'$ has the
required properties. 
\end{proof}

\section{Weinstein preliminaries}\label{sec:holonomy}

In this section we collect some facts about Weinstein structures
needed for the proofs of our main results. Most of the proofs are
omitted and we refer the reader to our book~\cite{CieEli12} for a more
systematic treatment of the subject. 

\subsection{Holonomy of Weinstein cobordisms}

In this subsection we consider Weinstein cobordisms $\fW=(W,\om,X,\phi)$
{\em without critical points} (of the function $\phi$). Its holonomy
along trajectories of $X$ defines a contactomorphism  
$$
   \Gamma_\fW:(\p_+W,\xi_+)\to (\p_-W,\xi_-)
$$
for the contact structures $\xi_\pm$ on $\p_\pm W$ induced by the
Liouville form $\lambda=i_X\om$. 

We say that two Weinstein structures $\fW=(\om,X,\phi)$ and $\wt\fW$
agree {\em up to scaling} on a subset $A\subset W$ if
$\wt\fW|_A=(C\om,X,\phi)$ for a constant $C>0$. Note
that in this case $\wt\fW|_A$ has Liouville form $C\lambda$.  

Let us fix a Weinstein cobordism $\ol\fW=(W,\ol\om,\ol X,\phi)$
without critical points. We denote by $\WW(\ol\fW)$ the space of all
Weinstein structures $\fW=(W,\om,X,\phi)$ with the same function
$\phi$ such that
\begin{itemize}
\item $\fW$ coincides with $\ol\fW$ on $\Op\p_-W$ and up to scaling on
  $\Op\p_+W$;
\item $\fW$ and $\ol\fW$ induce the same
  contact structures on level sets of $\phi$.
\end{itemize}  
Equivalently, $\WW(\ol\fW)$ can be viewed as the space of Liouville
forms $\lambda=f\bar\lambda+g\,d\phi$ with $f\equiv 1$ near $\p_-W$,
$f\equiv C$ near $\p_+W$, and $g\equiv 0$ near $\p W$, where
$\ol\lambda$ denotes the Liouville form of $\ol\fW$. 

Denote by $\DD(\ol\fW)$ the space of contactomorphisms
$(\p_+W,\xi_+)\to(\p_-W,\xi_-)$, where $\xi_\pm$ is the contact
structure induced on $\p_\pm W$ by $\ol\fW$.
Note that $\Gamma_\fW\in \DD(\ol\fW)$ for any $\fW\in \WW(\ol\fW)$.
The following lemma is the analogue of Lemma~\ref{lm:Morse-Serre}
in the context of Weinstein cobordisms.
 
\begin{lemma}\label{lm:Weinstein-Serre}
Let $\ol\fW$ be a Weinstein cobordism without critical points. Then
the map $\WW(\ol\fW) \to \DD(\ol\fW)$ that assigns to $\fW$ its
holonomy $\Gamma_\fW$ is a Serre fibration. In particular:

(i) Given $\fW\in \WW(\ol\fW)$ and an isotopy $h_t\in \DD(\ol\fW)$,
  $t\in[0,1]$, with $h_0=\Gamma_\fW$, there exists a path
  $\fW_t\in\WW(\ol\fW)$ with $\fW_0=\fW$ such that
  $\Gamma_{\fW_t}=h_t$ for all $t\in[0,1]$. 

(ii) Given a path $\fW_t\in\WW(\ol\fW)$, $t\in[0,1]$, and a path
  $h_t\in\DD(\ol\fW)$ which is homotopic to $\Gamma_{\fW_t}$ with
  fixed endpoints, there exists a path $\wt\fW_t\in\WW(\ol\fW)$ with
  $\wt \fW_0=\fW_0$ and $\wt \fW_1= \fW_1$ such that $\Gamma_{\wt
    \fW_t}= h_t$ for all $t\in[0,1]$.
\end{lemma}

\subsection{Weinstein structures near stable discs}
\label{sec:Weinstein-skel}  
 
The following two lemmas concern the construction of Weinstein
structures near stable discs of Smale cobordisms.  
\begin{lemma}\label{lm:prescribing-W-skel}
Let $\fS=(W,X,\phi)$ be an elementary Smale cobordism and $\om$ a
nondegenerate $2$-form on $W$. Let $D_1,\dots, D_k$ be the stable
discs of critical points of $\phi$, and set $\Delta:=\bigcup_{j=1}^kD_j$.
Suppose that the discs $D_1,\dots, D_k$ are $\om$-isotropic and the
pair $(\om,X)$ is Liouville on $\Op(\p_-W)$. Then for any
neighborhood $U$ of $\p_-W\cup\Delta$ there exists a homotopy
$(\om_t,X_t)$, $t\in[0,1]$, with the following properties:
\begin{enumerate}
\item $X_t$ is a gradient-like vector field for $\phi$ and $\om_t$ is
  a nondegenerate $2$-form on $W$ for all $t\in[0,1]$; 
\item $(\om_0,X_0)=(\om,X)$, and $(\om_t,X_t)=(\om,X)$ outside $U$ and
  on $\Delta\cup\Op(\p_-W)$ for all $t\in[0,1]$;
\item $(\om_1,X_1)$ is a Liouville structure on
  $\Op(\p_-W\cup\Delta)$.  
\end{enumerate}
\end{lemma}
 
Lemma~\ref{lm:prescribing-W-skel} has the following version for
homotopies. 

\begin{lemma}\label{lm:prescribing-W-skel-param}
Let $\fS_t=(W,X_t,\phi_t)$, $t\in[0,1]$, be an elementary Smale
homotopy and $\om_t$, $t\in[0,1]$, a family of 
nondegenerate $2$-forms on $W$. Let $\Delta_t$ be the 
union of the stable (half-)discs of zeroes of $X_t$
and set  
$$
   \Delta := \bigcup_{t\in[0,1]}\{t\}\times\Delta_t 
   \subset[0,1]\times W. 
$$
Suppose that $\Delta_t$ is $\om_t$-isotropic for all $t\in[0,1]$, the
pair $(\om_t,X_t)$ is Liouville on $\Op(\p_-W)$ for all $t\in[0,1]$,
and $(\om_0,X_0)$ and $(\om_1,X_1)$ are Liouville on all of $W$. 
Then for any open neighborhood $V=\bigcup_{t\in[0,1]}\{t\}\times V_t$ of 
$\Delta$ there exists an open neighborhood
$U=\bigcup_{t\in[0,1]}\{t\}\times U_t\subset V$ of $\Delta$ and a  
$2$-parameter family $(\om_t^s,X_t^s)$, $s,t\in[0,1]$, with the
following properties: 
\begin{enumerate}
\item $X_t^s$ is a gradient-like vector field for $\phi_t$ and
  $\om_t^s$ is a nondegenerate $2$-form on $W$ for all $s,t\in[0,1]$;  
\item $(\om^0_t,X^0_t)=(\om_t,X_t)$ for all $t\in[0,1]$,
  $(\om^s_0,X^s_0)=(\om_0,X_0)$ and $(\om^s_1,X^s_1)=(\om_1,X_1)$ for
  all $s\in[0,1]$, and
  $(\om^s_t,X^s_t)=(\om_t,X_t)$ outside $V_t$ and
  on $\Delta_t\cup\Op(\p_-W)$ for all $s,t\in[0,1]$;
\item $(\om^1_t,X^1_t)$ is a Liouville structure on
  $U_t$ for all $t\in[0,1]$.  
\end{enumerate}
\end{lemma}

\subsection{Weinstein homotopies}\label{ss:W-homotopies}

A smooth family $(W,\om_t,X_t,\phi_t)$, $t\in[0,1]$, of Weinstein
cobordism structures is called {\em Weinstein
homotopy} if the family $(W,X_t,\phi_t)$ is a Smale homotopy in the
sense of Section~\ref{sec:Smale-hom}. Recall that this means in
particular that the functions $\phi_t$ have $\p_\pm W$ as regular level
sets, and they are Morse except for finitely many $t\in(0,1)$ at which
a birth-death type critical point occurs. 

The definition of a Weinstein homotopy on a {\em manifold} $V$
requires more care. 
Consider first a smooth family $\phi_t:V\to\R$, $t\in[0,1]$, of
exhausting generalized Morse functions such that there exists a finite
set $A\subset(0,1)$ satisfying the conditions stated at the beginning
of Section~\ref{sec:Smale-hom}. 
We call $\phi_t$ a {\em simple Morse homotopy} if
there exists a sequence of smooth functions $c_1<c_2<\cdots$ on the
interval $[0,1]$ such that for each $t\in[0,1]$, $c_i(t)$ is a regular
value of the function $\phi_t$ and $\bigcup_k\{\phi_t\leq
c_k(t)\}=V$. A {\em Morse homotopy} 
is a composition of finitely many simple Morse homotopies. 
Then a {\em Weinstein homotopy} on the manifold $V$ is a family of
Weinstein manifold structures $(V,\omega_t,X_t,\phi_t)$ 
such that the associated functions $\phi_t$ form a Morse homotopy.

The main motivation for this definition of a Weinstein homotopy is the 
following result from~\cite{EliGro91} (see also~\cite{CieEli12}). 

\begin{prop}\label{prop:convex-stability}
If two Weinstein manifolds $\fW_0=(V,\omega_0,X_0,\phi_0)$ and
$\fW_1=(V,\omega_1,X_1,\phi_1)$ are Weinstein homotopic, then they are 
symplectomorphic. More precisely, there exists a diffeotopy $h_t:V\to
V$ with $h_0=\id$ such that $h_1^*\lambda_1-\lambda_0$ is exact, where
$\lambda_i=i_{X_i}\om_i$ are the Liouville forms. 
If $\fW_0$ and $\fW_1$ are the completions of homotopic Weinstein
domains, then we can achieve $h_1^*\lambda_1-\lambda_0=0$ outside a
compact set. 
\end{prop}

\begin{remark}
Without the hypothesis on the functions $c_k(t)$ in the definition of
a Weinstein homotopy, Proposition~\ref{prop:convex-stability} would
fail. Indeed, it is not hard to see that without this hypothesis 
all Weinstein structures on $\R^{2n}$ would be ``homotopic''. 
But according to McLean~\cite{McL09}, there are infinitely many
Weinstein structures on $\R^{2n}$ which are pairwise
non-symplectomorphic. 
\end{remark}

\begin{remark}\label{rem:Morse-homotopy}
It is not entirely obvious but true (see~\cite{CieEli12}) that any two
exhausting Morse functions on the same manifold can be connected by a
Morse homotopy. 
\end{remark}

The notion of Weinstein (or Stein) homotopy can be formulated in more
topological terms. Let us denote by $\Weinstein$ the space of
Weinstein structures on a fixed manifold $V$.
For any
$\fW_0\in\Weinstein$, $\eps>0$, $A\subset V$ compact, $k\in\N$, and
any unbounded sequence $c_1<c_2<\cdots$, we define the set
\begin{multline*}
   \UU(\fW_0,\eps,A,k,c) := \{ \fW=(\om,X,\phi)\in\Weinstein\mid\;
   \|\fW-\fW_0\|_{C^k(A)}<\eps,\\ 
   \ c_i \text{ regular values of }\phi\}. 
\end{multline*}
It is easy to see that these sets are the basis of a topology on
$\Weinstein$, 
and a smooth family of Weinstein structures satisfying the conditions
at the beginning of Section~\ref{sec:Smale-hom} 
defines a continuous path $[0,1]\to\Weinstein$ with respect to this
topology if and only if (possibly after target reparametrization of
the functions) it is a Weinstein homotopy according to the
definition above. A topology on the space $\Morse$ of exhausting
generalized Morse functions can be defined similarly. 

\subsection{Creation and cancellation of critical points of
  Weinstein structures}\label{sec:W-mod}

A key ingredient in Smale's proof of the $h$-cobordism theorem is the
creation and cancellation of pairs of critical points of a Morse
function. The following two propositions describe analogues of these
operations for Weinstein structures. 

\begin{prop}[Creation of critical points]\label{prop:creation}
Let $(W,\om,X,\phi)$ be a Weinstein cobordism without critical points.
Then given any point $p\in\Int W$ and any integer $k\in\{1,\dots,
n\}$, there exists a Weinstein homotopy $(\om,X_t,\phi_t)$ with the
following properties: 
\begin{enumerate}
\item $(X_0,\phi_0)=(X,\phi)$ and $(X_t,\phi_t)=(X,\phi)$ outside a
  neighborhood of $p$;  
\item $\phi_t$ is a creation family such that $\phi_1$ has a pair of
critical points of index $k$ and $k-1$. 
\end{enumerate}
\end{prop}

\begin{prop}[Cancellation of critical points]\label{prop:cancellation}
Let $(W,\om,X,\phi)$ be a Weinstein cobordism with exactly two
critical points $p,q$ of index $k$ and $k-1$, respectively, which are
connected by a unique gradient trajectory along
which the stable and unstable manifolds intersect transversely. 
Let $\Delta$ be the skeleton of $(W,X)$, i.e., the closure of the stable
manifold of the critical point $p$. Then there exists a
Weinstein homotopy $(\om,X_t,\phi_t)$ with the following properties:
\begin{enumerate}
\item $(X_0,\phi_0)=(X,\phi)$, and $(X_t,\phi_t)=(X,\phi)$ near $\p W$
and outside a neighborhood of $\Delta$; 
\item $\phi_t$ has no critical points outside $\Delta$;
\item $\phi_t$ is a cancellation family such that $\phi_1$ has no
critical points. 
\end{enumerate}
\end{prop}

\section{Existence and deformations of flexible Weinstein
  structures}\label{sec:ex} 

In this section we prove Theorem~\ref{thm:combined} from the
Introduction and some other results about flexible Weinstein manifolds
and cobordisms. For simplicity, we assume that individual functions
are Morse and $1$-parameter families are Morse homotopies in the sense
of Section~\ref{sec:Smale-hom}. The more general case of arbitrary
($1$-parameter families of) generalized Morse functions is treated
similarly. 

\subsection{Existence of Weinstein structures}\label{sec:MW-existence} 
 
The following two theorems imply Theorem~\ref{thm:combined}(a) from
the Introduction. 

\begin{thm}[Weinstein existence theorem]
\label{thm:Weinstein-existence}
Let $(W,\phi)$ be a $2n$-dimensional Morse cobordism such that $\phi$
has no critical points of index $>n$. Let $\eta$ be a nondegenerate
(not necessarily closed) $2$-form on $W$ and $Y$ a vector field near
$\p_-W$ such that $(\eta,Y,\phi)$ defines a Weinstein structure on
$\Op\p_-W$. Suppose that
either $n>2$, or $n=2$ and the contact structure induced by the
Liouville form $\lambda=i_Y\eta$ on $\p_-W$ is overtwisted.
Then there exists a Weinstein structure $(\om,X,\phi)$ on $W$ with the 
following properties: 
\begin{enumerate}
\item $(\om,X)=(\eta,Y)$ on $\Op \p_-W$;  
\item the nondegenerate $2$-forms $\om$ and $\eta$ on $W$ are homotopic
  {\rm rel} $\Op\p_-W$. 
\end{enumerate}
Moreover, we can arrange that $(\om,X,\phi)$ is flexible. 
\end{thm} 
 
Theorem~\ref{thm:Weinstein-existence} immediately implies the
following version for manifolds. 

\begin{thm}\label{thm:Weinstein-existence-man}
Let $(V,\phi)$ be a $2n$-dimensional manifold with an exhausting Morse
function $\phi$ that 
has no critical points of index $>n$. Let $\eta$ be a nondegenerate
(not necessarily closed) $2$-form on $V$.
Suppose that $n>2$.
Then there exists a Weinstein structure $(\om,X,\phi)$ on $V$ such
that the nondegenerate $2$-forms $\om$ and $\eta$ on $V$ are homotopic.
Moreover, we can arrange that $(\om,X,\phi)$ is flexible. \qed
\end{thm}

\begin{proof}[Proof of Theorem~\ref{thm:Weinstein-existence}] 
By decomposing the Morse cobordism  $\fM=(W,\phi)$ into elementary
ones, $W=W_1\cup\dots\cup W_N$, and inductively extending the
Weinstein structure over $W_1,\dots, W_N$, it suffices to consider the
case of an elementary cobordism. 
To simplify the notation, we will assume that $\phi$
has a unique critical point $p$; the general case is similar.  
Let us extend $Y$ to a gradient-like vector field for $\phi$ on $W$
and denote by $\Delta$ the stable disc of $p$.  

\begin{step1}
We first show that, after a homotopy of $(\eta,Y)$ fixed on
$\Op\p_-W$, we may assume that $\Delta$ is $\eta$-isotropic. 

The Liouville form $\lambda=i_Y\eta$ on $\Op\p_-W$ defines a 
contact structure $\xi:=\ker(\lambda|_{\p_-W})$ on $\p_-W$. We choose
an auxiliary $\eta$-compatible almost complex structure $J$ on $W$ which
preserves $\xi$ and maps $Y$ along $\p_-W$ to the Reeb vector field
$R$ of $\lambda|_{\p_-W}$. We apply
Theorem~\ref{thm:real-handle} to find a diffeotopy $f_t:W\to W$
such that the disc $\Delta'=f_1(\Delta)$ is $J$-totally real and
$J$-orthogonally attached to $\p_-W$. This is the only point in the
proof where the overtwistedness assumption for $n=2$ is needed.  
Moreover, according to Theorem~\ref{thm:real-handle}, in the case
$\dim\Delta=n$ we can  arrange that the Legendrian sphere $\p\Delta'$
in $(\p_-W,\xi)$ is loose (meaning that $\p_-W\setminus\p\Delta'$ is
overtwisted in the case $n=2$).  


Next we modify the homotopy $f_t^*J$ to keep it fixed near $\p_-W$. 
$J$-orthogonality implies that $\p\Delta'$ is tangent to the
maximal $J$-invariant distribution $\xi\subset T(\p_-W)$ and thus
$\lambda|_{\p\Delta'}=0$. Since the spaces $T\Delta'$ and
${\rm span}\{T\p\Delta',Y\}$ are both totally real and $J$-orthogonal to
$T(\p_-W)$, we can further adjust the disc $\Delta'$ (keeping
$\p\Delta'$ fixed) to make it tangent to $Y$ in a neighborhood of
$\p\Delta'$. 
It follows that we can modify $f_t$ such that it preserves the
function $\phi$ and the vector field $Y$ on a neighborhood $U$ of
$\p_-W$ (extend $f_t$ from $\p_-W$ to $U$ using the flow of $Y$). 

Hence, there exists a diffeotopy $g_t:W\to W$, $t\in[0,1]$, which
equals $f_t$ on $W\setminus U$, the identity on  $\Op\p_-W$, and
preserves $\phi$ (but not $Y$!) on $U$; see
Figure~\ref{fig:Delta}. Then the diffeotopy 
$k_t:=f_t^{-1}\circ g_t$ equals the identity on $W\setminus U$,
$f_t^{-1}$ on $\Op\p_-W$, and preserves $\phi$ on all of $W$.  
Thus the vector fields $Y_t:=k_t^*Y$ are gradient-like for
$\phi=k_t^*\phi$ and coincide with $Y$ on $(W\setminus
U)\cup\Op\p_-W$. The nondegenerate 2-forms $\eta_t:=g_t^*\eta$ are
compatible with $J_t:=g_t^*J$ and coincide with $\eta$ on
$\Op\p_-W$. Moreover, since $\Delta'$ is $J$-totally real, the
stable disc $\Delta_1:=k_1^{-1}(\Delta) = g_1^{-1}(\Delta')$ of $p$
with respect to $Y_1$ is $J_1$-totally real and $J_1$-orthogonally
attached to $\p_-W$. 

\begin{figure}
\centering
\includegraphics{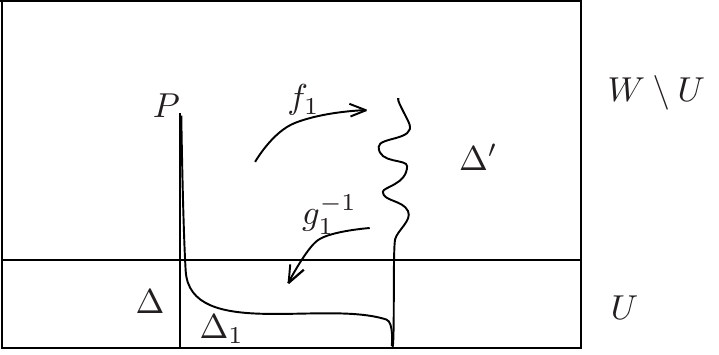}
\caption{Deforming the disc $\Delta$ to one which is totally real and
  $J$-orthogonally attached.}
\label{fig:Delta}
\end{figure}

After renaming $(\eta_1,Y_1,\Delta_1)$ back to $(\eta,Y,\Delta)$, we may
hence assume that $\Delta$ is $J$-totally real and $J$-orthogonally
attached to $\p_-W$ for some $\eta$-compatible almost complex structure
$J$ on $W$ which preserves $\xi$ and maps $Y$ to the Reeb vector field
$R$ along $\p_-W$. In particular, $\p\Delta$ is $\lambda$-isotropic
and $\Delta\cap\Op\p_-W$ is $\eta$-isotropic. Since the space of
nondegenerate 2-forms compatible with $J$ is contractible, after a further
homotopy of $\eta$ fixed on $\Op\p_-W$ and outside a neighborhood
of $\Delta$, we may assume that $\Delta$ is $\eta$-isotropic.  
\end{step1}

\begin{step2}
By Lemma~\ref{lm:prescribing-W-skel} there exists a 
homotopy $(\eta_t,Y_t)$, $t\in[0,1]$, of gradient-like vector fields
for $\phi$ and nondegenerate 2-forms on $W$, fixed on
$\Delta\cup\Op\p_-W$ and outside a neighborhood of $\Delta$, such
that $(\eta_0,Y_0)=(\eta,Y)$ and $(\eta_1,Y_1)$ is Liouville on
$\Op(\p_-W\cup\Delta)$. 
After renaming $(\eta_1,Y_1)$ back to $(\eta,Y)$, we may hence assume
that $(\eta,Y)$ is Liouville on a neighborhood $U$ of
$\p_-W\cup\Delta$. 
\end{step2}

\begin{step3}
Pushing down along trajectories of $Y$, we construct an 
isotopy of embeddings $h_t:W\into W$, $t\in[0,1]$, with $h_0=\id$ and
$h_t=\id$ on $\Op(\p_-W\cup\Delta)$, which preserves trajectories of
$Y$ and such that $h_1(W)\subset U$. Then
$(\eta_t,Y_t):=(h_t^*\eta,h_t^*Y)$ defines a homotopy of nondegenerate 
2-forms and vector fields on $W$, fixed on $\Op(\p_-W\cup\Delta)$,
from $(\eta_0,Y_0)=(\eta,Y)$ to the Liouville structure
$(\eta_1,Y_1)=:(\om,X)$. Since the $Y_t$ are proportional to $Y$, they
are gradient-like for $\phi$ for all $t\in[0,1]$. 

The Weinstein structure $(\om,X,\phi)$ will be flexible if we
choose the stable sphere $\p\Delta$ in Step 1 to be loose, so
Theorem~\ref{thm:Weinstein-existence} is proved. \qedhere  
\end{step3}
\end{proof}

\subsection{Homotopies of flexible Weinstein structures} 

The following Theorems~\ref{thm:W-hom1} and~\ref{thm:W-hom2} for
cobordisms, and Theorems~\ref{thm:W-hom1-man} and~\ref{thm:W-hom2-man}
for manifolds, are our main results concerning deformations of
flexible Weinstein structures. They imply
Theorem~\ref{thm:combined}(b) from the Introduction.

\begin{theorem}[First Weinstein deformation theorem]
\label{thm:W-hom1}
Let $\fW=(W,\om,X,\phi)$ be a flexible Weinstein cobordism of
dimension $2n$. Let $\phi_t$, $t\in[0,1]$, be a Morse homotopy
without critical points of index $>n$ with $\phi_0=\phi$ and
$\phi_t=\phi$ near $\p W$. 
In the case $2n=4$ assume that either $\p_-W$ is overtwisted, or
$\phi_t$ has no critical points of index $>1$. 
Then there exists a homotopy $\fW_t=(W,\om_t,X_t,\phi_t)$,
$t\in[0,1]$, of flexible Weinstein structures, starting at  
$\fW_0=\fW$, which is fixed near $\p_-W$ and fixed up to scaling near
$\p_+W$. 
\end{theorem}

\begin{thm}[Second Weinstein deformation theorem]
\label{thm:W-hom2}
Let $\fW_0=(\om_0,X_0,\phi_0)$ and $\fW_1=(\om_1,X_1,\phi_1)$ be two
flexible Weinstein structures on a cobordism $W$ of dimension $2n$. 
Let $\phi_t$, $t\in[0,1]$, be a Morse homotopy without critical points
of index $>n$ connecting $\phi_0$ and $\phi_1$.
In the case $2n=4$ assume that either $\p_-W$ is overtwisted, or
$\phi_t$ has no critical points of index $>1$. 
Let $\eta_t$, $t\in[0,1]$, be a homotopy of nondegenerate (not
necessarily closed) $2$-forms connecting $\om_0$ and $\om_1$ 
such that $(\eta_t,Y_t,\phi_t)$ is Weinstein near $\p_-W$ for a
homotopy of vector fields $Y_t$ on $\Op\p_-W$ connecting $X_0$ and
$X_1$.   

Then $\fW_0$ and $\fW_1$ can be connected by a homotopy
$\fW_t=(\om_t,X_t,\phi_t)$, $t\in[0,1]$, of flexible Weinstein
structures, agreeing with $(\eta_t,Y_t,\phi_t)$ on $\Op\p_-W$, such
that the paths of nondegenerate $2$-forms  
$t\mapsto\eta_t$ and $t\mapsto\om_t$, $t\in[0,1]$, are homotopic rel
$\Op\p_-W$ with fixed endpoints. 
\end{thm}

Theorems~\ref{thm:W-hom1} 
and~\ref{thm:W-hom2} will be proved in Sections~\ref{sec:proof-W1}
and~\ref{sec:proof-W2}. They have the following analogues 
for deformations of flexible Weinstein {\it manifolds}, which are
derived from the cobordism versions by induction over sublevel sets.

\begin{theorem}\label{thm:W-hom1-man}
Let $\fW=(V,\om,X,\phi)$ be a flexible Weinstein manifold of 
dimension $2n$. 
Let $\phi_t$, $t\in[0,1]$, be a Morse homotopy without critical points
of index $>n$ with $\phi_0=\phi$.  
In the case $2n=4$ assume that $\phi_t$ has no critical
points of index $>1$. 
Then there exists a homotopy $\fW_t=(V,\om_t,X_t,\phi_t)$,
$t\in[0,1]$, of flexible Weinstein structures such that $\fW_0=\fW$. 

If the Morse homotopy $\phi_t$ are fixed outside a compact set,
then the Weinstein homotopy $\fW_t$ can be chosen fixed outside a
compact set. \qed
\end{theorem}

\begin{thm}\label{thm:W-hom2-man}
Let $\fW_0=(\om_0, X_0,\phi_0)$ and $\fW_1=(\om_1, X_1,\phi_1)$ be two
flexible Weinstein structures on the same manifold $V$
of dimension $2n$. 
Let $\phi_t$, $t\in[0,1]$, be a Morse homotopy without critical points
of index $>n$ connecting $\phi_0$ and $\phi_1$. 
In the case $2n=4$, assume that $\phi_t$ has no critical points of
index $>1$. 
Let $\eta_t$ be a homotopy of nondegenerate $2$-forms on $V$
connecting $\om_0$ and $\om_1$. 
Then there exists a homotopy $\fW_t=(\om_t,X_t,\phi_t)$ of flexible
Weinstein structures connecting $\fW_0$ and $\fW_1$ such that the
paths $\om_t$ and $\eta_t$ of nondegenerate $2$-forms are homotopic
with fixed endpoints. \qed  
\end{thm}


\subsection{Proof of the first Weinstein deformation theorem}
\label{sec:proof-W1}

  
The proof of Theorem~\ref{thm:W-hom1} is based on
the following three lemmas. 
 
\begin{lemma}
\label{lm:W-interpolation}
Let $\fW=(W,\om,X,\phi)$ be a flexible Weinstein cobordism
and $Y$ a gradient-like vector field for $\phi$
such that the Smale cobordism $(W,Y,\phi)$ is {\em elementary}. 
Then there exists a family $X_t$, $t\in[0,1]$, of gradient-like vector
fields for $\phi$ and a family $\omega_t$, $t\in[0,\frac12]$, of symplectic
forms on $W$ such that 
\begin{itemize}
\item $\fW_t=(W,\om_t,X_t,\phi)$, $t\in[0,\frac12]$, is a Weinstein
  homotopy with $\fW_0=\fW$, fixed on $\Op\p_-W$ and fixed up to
  scaling on $\Op\p_+W$;  
\item $X_1=Y$ and the Smale cobordisms $(W,X_t,\phi)$,
  $t\in[\frac12,1]$, are elementary.  
\end{itemize}
\end{lemma}
 
\begin{proof}
{\bf Step 1.} Let $c_1<\dots<c_N$ be the critical values of the function $\phi$. 
Set $c_0:=\phi|_{\p_-W}$ and $c_{N+1}:=\phi|_{\p_+W}$.  
Choose $\eps\in(0,\min_{j=0,\dots,
  N}\frac{c_{j+1}-c_j}2)$ and define
\begin{align*}
   W_j &:= \{c_j-\eps\leq \phi\leq c_j+\eps\},\quad j=2,\dots, N-1, \\ 
   W_1 &:= \{\phi\leq c_1+\eps\},\qquad W_N := \{\phi\geq c_N-\eps\}, \\ 
   V_j &:= \{c_j+\eps\leq\phi\leq   c_{j+1}-\eps\},\quad j=1,\dots,
   N-1, \\
   \Sigma_j^{\pm } &:= \{\phi=c_{j}\pm\eps\}, \quad j=1,\dots, N;
   \end{align*}
see Figure~\ref{fig:VW}.

\begin{figure}
\centering
\includegraphics{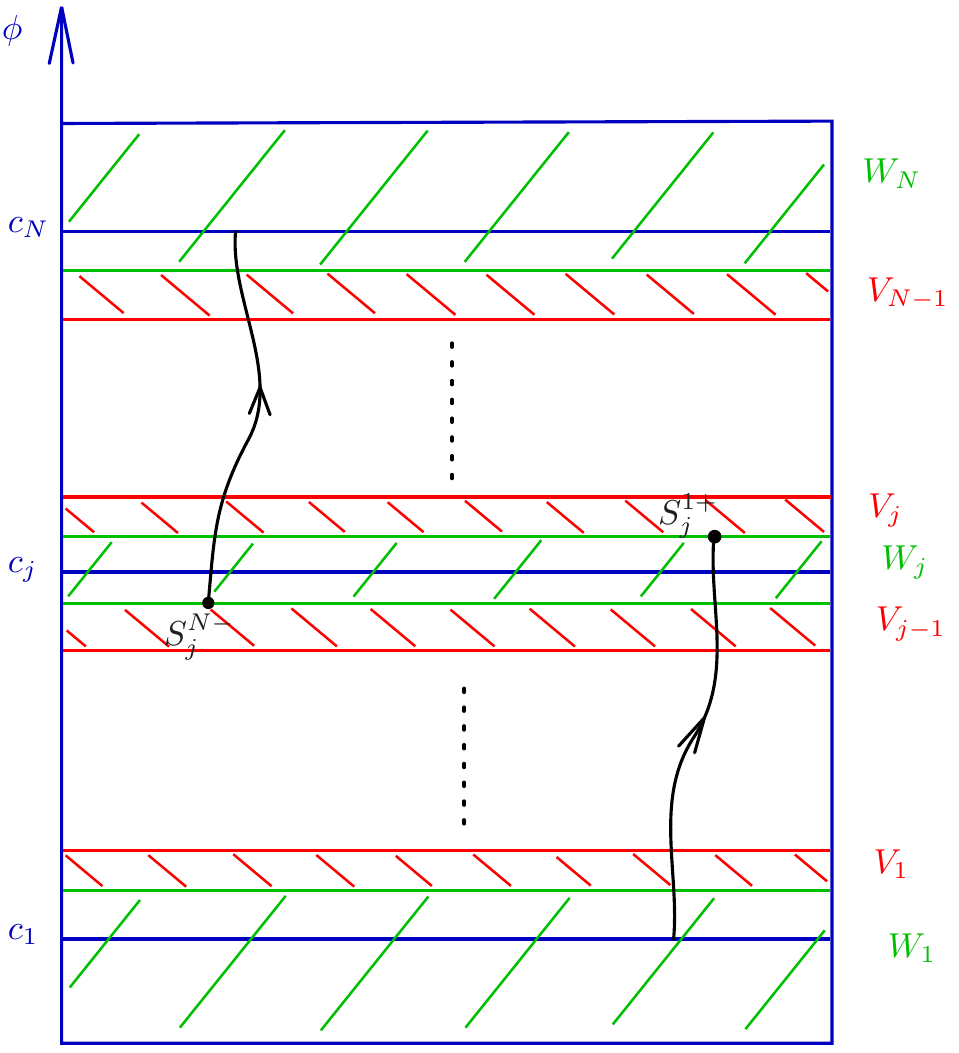}
\caption{The partition of $W$ into subcobordisms.}
\label{fig:VW}
\end{figure}

Thus  we have $\Sigma_j^+ = \p_-V_j=\p_+W_j$ for $j=1,\dots,N-1$ and
$\Sigma_j^- = \p_+V_{j} = \p_-W_j$ for $j=2,\dots, N$.  
We denote by $\xi_j^{\pm }$ the contact structure induced by the
Liouville form $i_X\om$ on $\Sigma_j^{\pm }$, $j=1,\dots, N$. 
 
For $k\geq j$ we denote by $S^{k-}_j$ the intersection of the union of
the $Y$-stable manifolds of the critical points on level
$c_k$ with the hypersurface $\Sigma^{-}_j$. Similarly, 
for $i\leq j$ we denote by $S^{i+}_j$ the intersection of the union of
the $Y$-unstable manifolds of the critical points on level
$c_i$ with the hypersurface $\Sigma^{+}_j$; see Figure~\ref{fig:VW}. 
Set 
$$
   \bS^-_j:=\bigcup\limits_{k\geq j} S^{k-}_j,\qquad
   \bS^+_j:=\bigcup\limits_{i\leq j} S^{i+}_j.
$$ 
The assumption that the Smale cobordism $(Y,\phi)$ is elementary
implies that $\bS^\pm_j$ is a union of spheres in
$\Sigma_j^\pm$. 
  
Consider on $\bigcup_{j=1}^{N}W_j$ the gradient-like vector fields
$Y_t:=(1-t)Y+tX$, $t\in[0,1]$, for $\phi$. Let us pick $\eps$ so small
that for all $t\in[0,1]$ the $Y_t$-unstable spheres in $\Sigma_j^+$ of
the critical points on level $c_j$ do not intersect the $Y$-stable
spheres in $\Sigma_j^+$ of any critical points on higher levels.  
By Lemma~\ref{lem:cone} we can extend the $Y_t$ to gradient-like
vector fields for $\phi$ on $W$ such that $Y_0=Y$ and $Y_t=Y$ outside 
$\Op\bigcup_{j=1}^{N }W_j$ for all $t\in[0,1]$.
By Lemma~\ref{lm:Morse-Serre}, 
this can be done in such a way that the intersection of the
$Y_t$-stable manifold of the critical point locus on level $c_i$ with
the hypersurface $\Sigma^{+}_j$ remains unchanged. This implies that the 
cobordisms $(W,Y_t,\phi)$ are elementary for all $t\in[0,1]$. 
After renaming $Y_1$ back to $Y$ and shrinking the $W_j$, we may hence
assume that $Y=X$ on $\Op\bigcup_{j=1}^{N }W_j$. Moreover, after
modifying $Y$ near $\p W$ we may assume that $Y=X$ on $\Op\p W$. 

We will construct the required homotopies $X_t$, $t\in[0,1]$, and
$\om_t$, $t\in[0,\frac12]$, separately on each $V_j$, $j=1,\dots,N-1$,
in such a way
that $X_t$ is fixed near $\p V_j$ for all $t\in[0,1]$ and $\omega_t$
is fixed up to scaling near $\p V_j$ for $t\in[0,\frac12]$. This will
allow us then to extend the homotopies $X_t$ and $\om_t$ to  
$\bigcup_{j=1}^{N }W_j$ as constant (resp.~constant up to scaling). 
%

\begin{step2}
Consider $V_j$ for $1\leq j\leq N-1$. To simplify the notation, we
denote the restriction of objects to $V_j$ by the 
same symbol as the original objects, omitting
the index $j$. Let us denote by $\XX(V_j,\phi)$
the space of all gradient-like vector fields for $\phi$ on $V_j$ that
agree with $X$ near $\p V_j$. We connect $X$ and $Y$ by the path
$Y_t:=(1-t)X+tY$ in $\XX(V_j,\phi)$. 

Denote by $\Gamma_{Y_t}:\Sigma_{j+1}^-\to\Sigma_j^+$ the holonomy of
the vector field $Y_t$ on $V_j$ and consider the isotopy
$g_t:=\Gamma_{Y_t}|_{\bS^-_{j+1}}: \bS^-_{j+1}\into\Sigma^+_j$. 
Suppose for the moment that {\it $\bS^-_{j+1}\subset\Sigma_{j+1}^-$ is
isotropic and loose} (this hypothesis will be satisfied below when we
perform induction on descending values of $j$).   

Since $\Gamma_{Y_0}=\Gamma_X$ is a contactomorphism, this implies that
the embedding $g_0$ is loose isotropic. Hence, by
Theorem~\ref{thm:h-isotropic-subcrit} for the subcritical case,
Theorem~\ref{thm:loose-3} for the Legendrian overtwisted case in
dimension $4$, and Theorem~\ref{thm:loose} in the Legendrian loose
case in dimension $2n>4$, the isotopy $g_t$ can be 
$C^0$-approximated by an isotropic isotopy. 
More precisely, there exists a $C^0$-small diffeotopy 
$\delta_t:\Sigma_j^+\to\Sigma_j^+$ with $ \delta_0=\Id$ such that 
$\delta_t\circ g_t$, $t\in[0,1]$, is loose isotropic with respect to the
contact structure $\xi_j^+$. 

The path $\Gamma_{Y_t}$, $t\in[0,1]$, in 
$\Diff(\Sigma_{j+1}^-,\Sigma_j^+)$ is homotopic with fixed endpoints
to the concatenation of the paths $\delta_t\circ\Gamma_{Y_t}$ (from
$\Gamma_{Y_0}$ to $\delta_1\circ\Gamma_{Y_1}$) and
$\delta_t^{-1}\circ\delta_1\circ \Gamma_{Y_1}$ (from
$\delta_1\circ\Gamma_{Y_1}$ to $\Gamma_{Y_1}$). Hence by
Lemma~\ref{lm:Morse-Serre} we find paths $Y'_t$ and $Y''_t$,
$t\in[0,1]$, in $\XX(V_j,\phi)$ such that 
\begin{itemize}
\item $Y'_0=X$, $Y'_1=Y''_0$ and $Y''_1=Y$;
\item $\Gamma_{Y'_t}=\delta_t\circ\Gamma_{Y_t}$ and $\Gamma_{Y''_t}=
  \delta_t^{-1}\circ\delta_1\circ \Gamma_{Y_1}$, $t\in[0,1]$. 
\end{itemize}
Note that $\Gamma_{Y'_t}|_{\bS^-_{j+1}}$ is loose isotropic. Moreover, by
choosing $\delta_t$ sufficiently $C^0$-small, we can ensure that
$\Gamma_{Y''_t}(\bS^-_{j+1})\cap \bS^+_{j} = \varnothing$ in
$\Sigma_j^+$ and $\Gamma_Y(\bS^-_{j+1})$ is loose in
$\Sigma_j^+\setminus \bS_j^+$. We extend the vector fields $Y'_t$ and
$Y''_t$ to $W$ by setting $Y'_t:=(1-t)X+tY$ and $Y''_t:=Y$ on
$W\setminus V_j$.  
The preceding discussion shows that the cobordisms $(W,Y''_t,\phi)$
are elementary for all $t\in[0,1]$. Hence it is sufficient to prove
the lemma with the original vector field $Y$ replaced by
$Y'_1=Y''_0$. To simplify the notation, we rename $Y'_1$ to $Y$
and the homotopy $Y'_t$ to $Y_t$. The new homotopy now has the
property that the isotopy $\Gamma_{Y_t}|_{\bS^-_{j+1}}: 
\bS^-_{j+1}\into\Sigma^+_j$ is loose isotropic and
$\Gamma_Y(\bS^-_{j+1})$ is loose in $\Sigma_j^+\setminus \bS_j^+$. 
So the image of $\Gamma_Y(\bS^-_{j+1})$ under the holonomy of the
elementary Weinstein cobordism $(W_j,\om,X=Y,\phi)$ is loose isotropic
in $\Sigma_j^-$. Since the union $S_j^-$ of the stable spheres of
$(W_j,Y)$ are loose by the flexibility hypothesis on $\fW$, this
implies that $\bS^-_j\subset\Sigma_j^-$ is loose isotropic. 

Now we perform this construction inductively in {\em descending} order
over $V_j$ for 
$j=N-1,\, N-2,\dots,1$, always renaming the new vector fields back to
$Y$. The resulting vector field $Y$ is then connected to $X$ by a
homotopy $Y_t$ such that the manifolds
$\bS^-_{j+1}\subset\Sigma_{j+1}^-$ and the isotopies
$\Gamma_{Y_t}|_{\bS^-_{j+1}}:\bS^-_{j+1}\into\Sigma^+_j$,
$t\in[0,1]$, are loose isotropic for all $j=1,\dots,N-1$. 
\end{step2}

\begin{step3}
Let $Y$ and $Y_t$ be as constructed in Step 2. Now we construct the
desired homotopies $X_t$ and $\om_t$ separately on each $V_j$,
$j=1,\dots,N-1$, keeping them fixed near $\p V_j$. 
We keep the notation from Step 2. By the contact isotopy extension
theorem, we can extend the isotropic isotopy 
$\Gamma_{Y_t}|_{\bS^-_{j+1}}:\bS^-_{j+1}\into\Sigma^+_j$ to a contact
isotopy $G_t:(\Sigma^-_{j+1},\xi^-_{j+1})\to (\Sigma^+_{j},\xi^+_{j})$
starting at $G_0=\Gamma_{Y_0}=\Gamma_X$.  
By Lemma~\ref{lm:Weinstein-Serre}, we find a Weinstein homotopy $\wt
\fW_t=(V_j,\wt\om _t,\wt X_t,\phi )$ beginning at $\wt\fW_0=\fW$
with holonomy $\Gamma_{\wt\fW_t}=G_t$ for all $t\in[0,1]$.    
%
%
Now Lemma~\ref{lm:fixing-A-Morse} provides a path
$X_t\in\XX(V_j,\phi)$ such that  
\begin{enumerate}
\item $X_t=\wt X_{2t}$ for $t\in[0,\frac12]$;
\item $X_1=Y_1=Y$;
\item $\Gamma_{X_t}(\bS^-_{j+1})=\Gamma_{Y}(\bS^-_{j+1})$ for
  $t\in[\frac12,1]$. 
\end{enumerate}
Over the interval $[0,\frac12]$ the Smale homotopy $\fS_t =(V_j,
X_t,\phi)$ can be lifted to the Weinstein homotopy $\fW_t=(V_j,\om_t,
X_t,\phi)$, where $\om_t:=\wt\om_{2t}$. 
  
Condition (iii) implies that $\Gamma_{X_t}(\bS_{j+1}^{-})\cap\bS_j^{+}
= \varnothing$ for all $t\in[\frac12,1]$, so the resulting Smale
homotopy on $W$ is elementary over the interval
$[\frac12,1]$. \qedhere 
\end{step3}
\end{proof}
 
The following lemma is the analogue of Lemma~\ref{lm:W-interpolation}
in the case that the Smale cobordism $(W,Y,\phi)$ is not elementary,
but has exactly two critical points connected by a unique trajectory. 

\begin{lemma}
\label{lm:W-interpolation-II}
Let $\fW=(W,\om,X,\phi)$ be a flexible Weinstein cobordism
and $Y$ a gradient-like vector field for $\phi$. Suppose that 
the function $\phi$ has exactly two critical points connected by a
unique $Y$-trajectory along which the stable and unstable manifolds
intersect transversely.  
Then there exists a family $X_t$, $t\in[0,1]$, of gradient-like vector
fields for $\phi$ and a family $\omega_t$, $t\in[0,\frac12]$, of symplectic
forms on $W$ such that 
\begin{itemize}
\item $\fW_t=(W,\om_t,X_t,\phi)$, $t\in[0,\frac12]$, is a homotopy
  with $\fW_0=\fW$, fixed on $\Op\p_-W$ and fixed up to scaling on
  $\Op\p_+W$;  
\item $X_1=Y$ and for $t\in[\frac12,1]$ the critical points of the
  function $\phi$ are connected by a unique $X_t$-trajectory. 
\end{itemize}
\end{lemma}
 
\begin{proof}
Let us denote the critical points of the function $\phi$ by $p_1$ and
$p_2$ and the corresponding critical values by $c_1<c_2$. As in the
proof of Lemma~\ref{lm:W-interpolation}, for sufficiently small
$\eps>0$, we split the cobordism $W$ into three parts:  
\begin{gather*}
   W_1:=\{\phi\leq c_1+\eps\},\quad
   V:=\{c_1+\eps\leq \phi\leq c_2-\eps\},\quad 
   W_2:=\{\phi\geq c_2-\eps\}. 
\end{gather*} 
Arguing as in Step 1 of the proof of Lemma~\ref{lm:W-interpolation}, we
reduce to the case that $Y=X$ on $\Op(W_1\cup W_2)$.   

On $V$ consider the gradient-like vector fields $Y_t:=(1-t)X+tY$ for
$\phi$. Let $\Sigma :=\{\phi=c_1+\eps\}=\p_-V$. 
Denote by $S_t\subset\Sigma$ the $Y_t$-stable sphere of $p_2$ and by
$S^+\subset\Sigma$ the $Y$-unstable sphere of $p_1$. Note that $S^+$
is coisotropic, $S_0$ is isotropic, and $S_1$ intersects $S^+$
transversely in a unique point $q$. We deform $S_1$ to $S_1'$ by a
$C^0$-small deformation, keeping the unique transverse intersection
point $q$ with $S^+$, such that $S_1'$ is isotropic near $q$. Connect
$S_0$ to $S_1'$ by an isotopy $S_t'$ which is $C^0$-close to $S_t$. 
Due to the flexibility hypothesis on $\fW$, the isotropic sphere
$S_0'=S_0$ is loose. Hence by Theorems~\ref{thm:h-isotropic-subcrit},
~\ref{thm:loose-3},
 and~\ref{thm:loose} 
we can $C^0$-approximate $S_t'$ by an isotropic isotopy $\wt S_t$ such
that $\wt S_0=S_0'=S_0$, and $\wt S_1$ coincides with $S_1'$ near
$q$. In particular, $\wt S_1$ has $q$ as the unique transverse
intersection point with $S^+$.   
Arguing as in Steps 2 and 3 of the proof of
Lemma~\ref{lm:W-interpolation},  we now construct a Weinstein homotopy 
$\fW_t=(V,\om_t, X_t,\phi)$, $t\in[0,\frac12]$, fixed near $\p_-V$ and
fixed up to scaling near $\p_+V$, and Smale cobordisms
$(V,X_t,\phi)$, $t\in[\frac12,1]$, fixed near $\p V$, such that 
\begin{itemize}
\item $\fW_0=\fW|_{V}$ and $X_1=Y|_{V}$;
\item the $X_t$-stable sphere of $p_2$ in $\Sigma$ equals $\wt S_{2t}$ for
  $t\in[0,\frac12]$, and $\wt S_1$ for $t\in[\frac12,1]$.
\end{itemize}
In particular, for $t\in[\frac12,1]$ the $X_t$-stable sphere of $p_2$
in $\Sigma$ intersects $S^+$ transversely in the unique point $q$, so  
the two critical points $p_1,p_2$ are connected by a unique
$X_t$-trajectory for $t\in[\frac12,1]$.  
\end{proof}

The following lemma will serve as induction step in proving
Theorem~\ref{thm:W-hom1}. 

\begin{lemma}\label{lm:W-def-elem}
Let $\fW=(W,\om,X,\phi)$ be a flexible Weinstein cobordism of
dimension $2n$. Let $\fS_t=(W,Y_t,\phi_t)$, $t\in[0,1]$, be an {\em
  elementary} Smale homotopy without critical points of index $>n$
such that $\phi_0=\phi$ on $W$ and $\phi_t=\phi$ near $\p W$ (but not
necessarily $Y_0=X$!).  
If $2n=4$ and $\fS_t$ is of Type IIb assume that either $\p_-W$ is
overtwisted, or $\phi_t$ has no critical points of index $>1$. 
Then there exists a homotopy $\fW_t=(W,\om_t,X_t,\phi_t)$,
$t\in[0,1]$, of flexible Weinstein structures, starting at  
$\fW_0=\fW$, which is fixed near $\p_-W$ and fixed up to scaling near
$\p_+W$. 
\end{lemma}

\begin{proof}
{\bf Type I. } 
Consider first the case when the homotopy $\fS_t$ is elementary of
Type I. We point out that $(W,X,\phi)$ need not be elementary. To
remedy this, we apply Lemma~\ref{lm:W-interpolation} to construct
families $X_t$ and $\om_t$ such that 
\begin{itemize}
\item $\fW_t=(W,\om_t,X_t,\phi)$, $t\in[0,\frac12]$, is a 
  Weinstein homotopy with $\fW_0=\fW$, fixed on $\Op\p_-W$ and fixed
  up to scaling on $\Op\p_+W$;  
\item $X_1=Y_0$ and the Smale cobordisms $(W,X_t,\phi)$,
  $t\in[\frac12,1]$, are elementary.  
\end{itemize}
Thus it is sufficient to prove the lemma for the Weinstein cobordism
$(\om_{\frac12},X_{\frac12},\phi)$ instead of $\fW$, and the
concatenation of the Smale homotopies $(X_t,\phi)_{t\in[\frac12,1]}$
and $(Y_t,\phi_t)_{t\in[0,1]}$ instead of $(Y_t,\phi_t)$.   
To simplify the notation, we rename the new Weinstein cobordism and
Smale homotopy back to $\fW=(\om,X,\phi)$ and $(Y_t,\phi_t)$. So in
the new notation we now have $X=Y_0$.  
 
According to Lemma~\ref{lem:prescribe-profile}
there exists a family $\wt\phi_t$, $t\in[0,1]$, of Lyapunov functions for
$X$ with the same profile as the family $\phi_t$, and such that
$\wt\phi_0=\phi$ and $\wt\phi_t=\phi_t $ on $\Op\p W$.  
Then Lemma~\ref{lm:Morse-profiles} 
provides a diffeotopy $h_t:W\to W$, $t\in[0,1]$, such that $h_0=\Id$,
$h_t|_{\Op\p W}=\Id$, and $\phi_t=\wt\phi_t\circ h_t$ for all
$t\in[0,1]$. Thus the Weinstein homotopy $(W,\om_t=h_t^*\om, X_t=
h_t^*X, \phi_t=h_t^*\wt\phi_t)$, $t\in[0,1]$, has the desired
properties. It is flexible because the $X_t$-stable spheres in $\p_-W$
are loose for $t=0$ and moved by an isotropic isotopy, so they remain
loose for all $t\in[0,1]$. 
 
{\bf Type IId. }
Suppose now that the homotopy $\fS_t$ is of Type IId. Let
$t_0\in[0,1]$ be the parameter value for which the function $\phi_t$
has a death-type critical point. In this case the function $\phi$ has
exactly two critical points $p$ and $q$ connected by a unique
$Y_0$-trajectory. Arguing as in the Type I case, using
Lemma~\ref{lm:W-interpolation-II} instead of
Lemma~\ref{lm:W-interpolation}, we can again reduce to the case that
$X=Y_0$. 
       
Then Proposition~\ref{prop:cancellation} provides an elementary
Weinstein homotopy $(W,\om,\wt X_t,\wt\phi_t)$ of Type IId
starting from $\fW$ and killing the critical points $p$ and $q$ at
time $t_0$. One can also arrange that $(\wt X_t,\wt\phi_t)$ coincides
with $(X,\phi)$ on $\Op\p W$, and (by composing $\wt\phi_t$ with
suitable functions $\R\to\R$) that the homotopies $\wt\phi_t$ and
$\phi_t$ have equal profiles. Then Lemma~\ref{lm:Morse-profiles} 
provides a diffeotopy $h_t:W\to W$, $t\in[0,1]$, such that $h_0=\Id$,
$h_t|_{\Op\p W}=\Id$, and $\phi_t=\wt\phi_t\circ h_t$ for all
$t\in[0,1]$. Thus the Weinstein homotopy  $(W,\om_t=h_t^*\om, X_t=
h_t^*\wt X, \phi_t=h_t^*\wt\phi_t)$, $t\in[0,1]$, has the desired
properties. It is flexible because the intersections of the
$X_t$-stable manifolds with regular level sets remain
loose for $t\in[0,t_0]$ and there are no critical points for $t>t_0$. 

{\bf Type IIb. }The argument in the case of Type IIb is similar,
except that we use Proposition~\ref{prop:creation} instead of
Proposition~\ref{prop:cancellation} and we do not need a preliminary 
homotopy.
However, the flexibility of $\fW_t$ for $t\geq t_0$ requires an
additional argument. 

Consider first the case $2n>4$. 
Suppose $\phi_1$ has critical points $p$ and $q$ of index $n$ and
$n-1$, respectively (if they have smaller indices flexibility is
automatic). Then the closure $\Delta$ of the $X_1$-stable
manifold of the point $p$ intersects $\p_-W$ along a Legendrian disc
$\p_-\Delta$
The boundary $S_q^-$
of this disc is the intersection with $\p_-W$ of the $X_1$-stable
manifold $D_q^-$ of $q$. According to Remark~\ref{rem:loose-knots}(1)
all Legendrian discs are loose, or more precisely, 
$\p_-\Delta\setminus S_q^-$ is loose in $\p_-W\setminus S_q^-$. 
Let $c$ be a regular value of $\phi_1$ which separates
$\phi_1(q)$ and $\phi_1(p)$ and consider the level set
$\Sigma:=\{\phi_1=c\}$. Flowing along $X_1$-trajectories
defines a contactomorphism $\p_-W\setminus
S_q^-\to\Sigma\setminus D_q^+$ mapping $\p_-\Delta\setminus
S_q^-$ onto $\Delta\cap\Sigma\setminus\{r\}$, where $r$ is the
unique intersection point of $\Delta$ and the $X_1$-unstable manifold
$D_q^+$ in the level set $\Sigma$. It follows that
$\Delta\cap\Sigma\setminus\{r\}$ is loose in $\Sigma\setminus\{r\}$,
and hence $\Delta\cap\Sigma$ is loose in $\Sigma$. 
This proves flexibility of $\fW_1$, and thus of $\fW_t$ for $t\geq
t_0$. 

Finally, consider the case $2n=4$. If the critical points have indices
$1$ and $0$, flexibility is automatic. If they have indices $2$ and $1$
and $\p_-W$ is overtwisted, we can arrange that
$\p_-\Delta\subset\p_-W$ (in the notation above) has an
overtwisted disc in its complement, hence so does the intersection
of $\Delta$ with the regular level set $\{\phi=c\}$.  
\end{proof} 
\medskip
  
\begin{proof}[Proof of Theorem~\ref{thm:W-hom1}]
Let us pick gradient-like vector fields $Y_t$ for $\phi_t$ with
$Y_0=X$ and $Y_t=X$ near $\p W$ to get a Smale homotopy
$\fS_t=(W,Y_t,\phi_t)$, $t\in[0,1]$.  
By Lemma~\ref{lm:admiss-homotopy} we find an admissible partition for
the Smale homotopy $\fS_t$. Thus we get a sequence
$0=t_0<t_1<\dots<t_p=1$ of parameter values and smooth families of
partitions 
$$
   W = \bigcup\limits_{j=1}^{N_k} W^k_j(t), \qquad W^k_j(t):=\{c^k_{j-1}(t)\leq
   \phi_t\leq c^k_j(t)\}, \qquad t\in[t_{k-1},t_k],
$$
such that each Smale homotopy
$$
   \fS^k_j := \left(W^k_j(t),
     Y_t|_{W^k_j(t)},\phi_t|_{W^k_j(t)}\right)_{t\in[t_{k-1},t_k]}  
$$
is elementary.
We will construct the Weinstein homotopy $(\om_t,X_t,\phi_t)$ on the
cobordisms $\bigcup_{t\in[t_{k-1},t_k]}W^k_j(t)$ inductively over
$k=1,\dots, p$, and for fixed $k$ over $j=1,\dots, N_k$. 
  
Suppose the required Weinstein homotopy is already constructed on $W$
for $t\leq t_{k-1}$. 
To simplify the notation we rename $\phi_{t_{k-1}}$ to $\phi$, the
vector fields $X_{t_k}$ and $Y_{t_k}$ to $X$ and $Y$, and the
symplectic form $\om_{t_{k-1}}$ to $\om$. We also write $N$ instead of
$N_k$, $W_j$ and $W_j(t)$ instead of $W^k_j(t_{k-1})$ and $W^k_j(t)$,
and replace the interval $[t_{k-1},t_k]$ by $[0,1]$. 

There exists a diffeotopy $f_t:W\to W$, fixed on $\Op\p W$, with
$f_0=\Id$ and such that $f_t(W_j )=W_j(t)$ for all
$t\in[0,1]$. Moreover, we can choose $f_t$ and a diffeotopy
$g_t:\R\to\R$ with $g_0=\id$ such that the function $\wh
\phi_t:=g_t\circ\phi_t\circ f_t$ coincides with $\phi$ on $\Op\p W_j$
for all $t\in[0,1]$, $j=1,\dots,N$.
Set $\wh Y_t:=f_t^*Y_t$. 
So we have a flexible Weinstein cobordism
$\fW=(W=\bigcup_{j=1}^NW_j,\om,X,\phi=\wh\phi_0)$ and a Smale homotopy $(\wh
Y_t,\wh\phi_t)$, $t\in[0,1]$, whose restriction to each $W_j$ is
elementary. (But the restriction of $\fW$ to $W_j$ need not be
elementary.)  
 
Now we apply Lemma~\ref{lm:W-def-elem} inductively for $j=1,\dots, N$
to construct Weinstein homotopies $\wh\fW^j_t=(W_j,\wh\om_t,\wh
X_t,\wh\phi_t)$, fixed near $\p_-W_j$ and fixed up to scaling near
$\p_+W_j$, with $\wh\fW^j_0=\fW|_{W_j}$. Thus the $\fW^j_t$ fit
together to form a Weinstein homotopy $\wh\fW_t=(\wh\om_t,\wh
X_t,\wh\phi_t)$ on $W$. The desired Weinstein homotopy on $W$ is now
given by 
\begin{equation*}
   \fW_t := \left(f_{t*}\wh\om_t,f_{t*}\wh
     X_t,g_t^{-1}\circ\wh\phi_t\circ f_t^{-1}\right). \qedhere
\end{equation*}
\end{proof}

\subsection{Proof of the second Weinstein deformation theorem}
\label{sec:proof-W2}

Let us extend the vector fields $Y_t$ from $\Op\p_-W$ to a path of
gradient-like vector fields  for $\phi_t$ on $W$ connecting $X_0$ and
$X_1$. 
We will deduce Theorem~\ref{thm:W-hom2} from
Theorem~\ref{thm:W-hom1} and the following special
case, which is just a 1-parametric version of the Weinstein
Existence Theorem~\ref{thm:Weinstein-existence}. 

\begin{lemma}\label{lem:Weinstein-homotopy-elem}
Theorem~{\rm\ref{thm:W-hom2}} holds under the
additional hypothesis that $\phi_t=\phi$ is independent of $t\in[0,1]$
and the Smale homotopy $(W,Y_t,\phi)$ is elementary.   
\end{lemma} 

\begin{proof} 
The proof is just a 1-parametric version of the proof of
Theorem~\ref{thm:Weinstein-existence}, using 
Theorem~\ref{thm:real-handle-subcrit} and
Lemma~\ref{lm:prescribing-W-skel-param}  
instead of Theorem~\ref{thm:real-handle} and
Lemma~\ref{lm:prescribing-W-skel}.  
\end{proof}

\begin{lemma}\label{lem:W-hom-coarse}
Theorem~{\rm\ref{thm:W-hom2}} holds under the
additional hypothesis that $\phi_t=\phi$ is independent of
$t\in[0,1]$. 
\end{lemma}

\begin{proof}
Let us pick regular values 
$$ 
   \phi|_{\p_-W} = c_0 < c_1 <\dots<c_N = \phi|_{\p_+W}
$$
such that each $(c_{k-1},c_k)$ contains at most one critical value. Then the
restriction of the homotopy $(Y_t,\phi)$, $t\in[0,1]$, to each cobordism   
$W^k :=  \{c_{k-1} \leq \phi\leq c_k \}$ is elementary.  

We apply Lemma~\ref{lem:Weinstein-homotopy-elem} to the restriction of
the homotopy $(\eta_t,Y_t,\phi)$ to $W^1$. Hence $\fW_0|_{W^1}$ and
$\fW_1|_{W^1}$ are connected by a homotopy $\fW_t^1=(\om_t^1,X_t^1,\phi)$,
$t\in[0,1]$, of flexible Weinstein structures on $W^1$, agreeing with
$(\eta_t,Y_t,\phi_t)$ on $\Op\p_-W$, such that the paths
$t\mapsto\om_t^1$ and $t\mapsto\eta_t$, $t\in[0,1]$, of nondegenerate
2-forms on $W^1$ are connected by a homotopy $\eta_t^s$, $s,t\in[0,1]$
rel $\Op\p_-W$ with fixed endpoints. We use the homotopy $\om_t^s$ to
extend $\om_t^1$ to nondegenerate 2-forms $\eta_t^1$ on $W$ such that
$\eta_0^1=\om_0$, $\eta_1^1=\om_1$, $\eta_t^1=\eta_t$ outside a
neighborhood of $W^1$, and the paths $t\mapsto\eta_t^1$ and
$t\mapsto\eta_t$, $t\in[0,1]$, of nondegenerate 2-forms on $W$ are
homotopic rel $\Op\p_-W$ with fixed endpoints. 
By Lemma~\ref{lem:cone}, we can extend $X_t^1$ to gradient-like vector
fields $Y_t^1$ for $\phi$ on $W$ such that $Y_0^1=X_0$ and
$Y_1^1=X_1$. 
Now we can apply Lemma~\ref{lem:Weinstein-homotopy-elem} to the
restriction of the homotopy $(\eta_t^1,Y_t^1,\phi)$ to the elementary
cobordism $W^2$ and continue inductively to construct homotopies
$(\eta_t^k,Y_t^k,\phi)$ on $W$ which are Weinstein on $W^k$, so
$(\eta_t^N,Y_t^N,\phi)$ is the desired Weinstein homotopy. Note that
$(\eta_t^N,Y_t^N,\phi)$ is flexible because its restriction to each
$W^k$ is flexible.  
\end{proof} 

\begin{proof}[Proof of Theorem~\ref{thm:W-hom2}]
Let us reparametrize the given homotopy 
$(\eta_t,Y_t,\phi_t)$,
$t\in[0,1]$, to make it constant for $t\in[\frac12,1]$.    
After pulling back $(\eta_t,Y_t,\phi_t)$ by a diffeotopy and target
reparametrizing $\phi_t$, we may further assume that $\phi_t$ is
independent of $t$ on $\Op\p W$. 

By Theorem~\ref{thm:W-hom1}, $\fW_0$ can be
extended to a homotopy $\fW_t=(\om_t,X_t,\phi_t)$,
$t\in[0,\frac12]$, of flexible Weinstein structures on $W$, fixed on
$\Op\p_-W$. We can modify $\fW_t$ 
to make it agree with $(\eta_t,Y_t,\phi_t)$ on $\Op\p_-W$. Note that
$\fW_{\frac12}$ and $\fW_1$ share the same function $\phi_{\frac12}=\phi_1$.   
We connect $\om_{\frac12}$ and $\om_1$ by a path $\eta_t'$,
$t\in[\frac12,1]$ of nondegenerate 2-forms by following the path $\om_t$
backward and then $\eta_t$ forward. Since $\om_t=\eta_t$ on
$\Op\p_-W$ for $t\in[0,\frac12]$, we can modify the path $\eta_t'$ to
make it constant equal to $\om_{\frac12}=\om_1$ on $\Op\p_-W$.  
By Lemma~\ref{lem:cone}, we can connect $X_{\frac12}$ and $X_1$ by a
homotopy $Y_t'$, $t\in[\frac12,1]$, of gradient-like vector fields for
$\phi_1$ which agree with $X_{\frac12}=X_1$ on $\Op\p_-W$.  

So we can apply Lemma~\ref{lem:W-hom-coarse} to the homotopy
$(\eta_t',Y_t',\phi_1)$, $t\in[\frac12,1]$. Hence $\fW_{\frac12}$ and
$\fW_1$ are connected by a homotopy $\fW_t=(\om_t,X_t,\phi_1)$,
$t\in[\frac12,1]$, of flexible Weinstein structures, agreeing with
$(\om_1,X_1,\phi_1)$ on $\Op\p_-W$, such that the paths of
nondegenerate 2-forms $t\mapsto\om_t$ and $t\mapsto\eta_t'$,
$t\in[\frac12,1]$, are homotopic rel $\Op\p_-W$ with fixed
endpoints. It follows from the definition of $\eta_t'$ that the
concatenated path $\om_t$, $t\in[0,1]$, is homotopic to $\eta_t$,
$t\in[0,1]$. Thus the concatenated Weinstein homotopy $\fW_t$,
$t\in[0,1]$, has the desired properties.  
\end{proof}

\section{Applications}\label{sec:app}

\subsection{The Weinstein $h$-cobordism theorem}

Most of your applications are based on the following result, which is
a direct consequence of the ``two-index 
theorem'' of Hatcher and Wagoner; see~\cite{CieEli12} for its formal
derivation from the results in~\cite{HatWag73,Igu88}. 

\begin{thm}\label{thm:pseudo}
Any two Morse functions without critical points of index $>n$ on a
cobordism or a manifold of dimension $2n>4$ can be connected by a
Morse homotopy without critical points of index $>n$ (where, as usual,
functions on a cobordism $W$ are required to have $\p_\pm W$ as
regular level sets and functions on a manifold are required to be 
exhausting).  
\end{thm}

\begin{cor}\label{cor:pseudo}
In the case $2n>4$, we can remove the hypothesis on
the existence of a Morse homotopy $\phi_t$ from Theorems~\ref{thm:W-hom1},
\ref{thm:W-hom2}, \ref{thm:W-hom1-man} and \ref{thm:W-hom1-man} and
still conclude the existence of the stated Weinstein homotopies. \qed
\end{cor}

In particular, we have the following Weinstein version of the
$h$-cobordism theorem.  

\begin{cor}[Weinstein $h$-cobordism theorem]
\label{cor:h-cobW}
Any flexible Weinstein structure on a product cobordism
$W=Y\times[0,1]$ of dimension $2n>4$ is homotopic to a Weinstein
structure $(W,\om, X,\phi)$, where $\phi:W\to[0,1]$ is a function
without critical points. \hfill$\Box$
\end{cor}

\subsection{Symplectomorphisms of flexible Weinstein manifolds}
 
Theorem~\ref{thm:W-hom2-man} has the following consequence for
symplectomorphisms of flexible Weinstein manifolds. 

\begin{theorem}\label{thm:symplectomorphism}
Let $\fW=(V,\om,X,\phi)$ be a flexible Weinstein manifold of dimension
$2n>4$, and $f:V\to
V$ a diffeomorphism such that $f^*\om$ is homotopic to $\om$ through
nondegenerate $2$-forms. Then there exists a diffeotopy $f_t:V\to V$,
$t\in[0,1]$, such that $f_0=f$, and $f_1$ is an exact
symplectomorphism of $(V,\om)$.   
\end{theorem} 

\begin{proof}
By Theorem~\ref{thm:W-hom2-man} and Corollary~\ref{cor:pseudo}, there
exists a Weinstein homotopy 
$\fW_t$ connecting $\fW_0=\fW$ and $\fW_1=f^*\fW$. Thus
Proposition~\ref{prop:convex-stability} provides a diffeotopy $h_t:V\to
V$ such that $h_0=\id$ and $h_1^*f^*\lambda-\lambda$ is exact, where
$\lambda$ is the Liouville form of $\fW$. Now $f_t=f\circ h_t$ is the
desired diffeotopy.  
\end{proof}

\begin{remark}
Even if $\fW$ is of finite type and $f=\id$ outside a compact set, the
diffeotopy $f_t$ provided by Theorem~\ref{thm:symplectomorphism}
will in general {\em not} equal the identity outside a compact set.  
\end{remark}

\subsection{Symplectic pseudo-isotopies}
\label{sec:symplectic-pseudo-isotopies}
 

Let us recall the basic notions of pseudo-isotopy theory
from~\cite{Cer70,HatWag73}. For a manifold $W$ (possibly 
with boundary) and a closed subset $A\subset
W$, we denote by $\Diff(W,A)$ the space of diffeomorphisms of $W$ fixed
on $\Op(A)$, equipped with the $C^\infty$-topology. For a cobordism
$W$, the restriction map to $\p_+W$ defines a fibration 
$$
   \Diff(W,\p W)\to \Diff(W,\p_-W)\to \Diff_\PP(\p_+W), 
$$
where $\Diff_\PP(\p_+W)$ denotes the image of the restriction map
$\Diff(W,\p_-W)\to \Diff(\p_+W)$. 
For the product cobordism $I\times M$, $I=[0,1]$, $\p M=\varnothing$,
$$
   \PP(M):=\Diff(I\times M,0\times M)
$$
is the group of 
{\em pseudo-isotopies} 
\index{pseudo-isotopy} 
of $M$. Denote by $\Diff_\PP(M)$
the group of diffeomorphisms of $M$ that are {\em pseudo-isotopic to
  the identity}, i.e., that appear as the restriction to $1\times M$ of
an element in $\PP(M)$. Restriction to $1\times
M$ defines the fibration
$$
   \Diff(I\times M,\p I\times M)\to \PP(M)\to \Diff_\PP(M),
$$
and thus a homotopy exact sequence
$$
   \cdots\to\pi_0\Diff(I\times M,\p I\times M)\to \pi_0\PP(M)\to
   \pi_0\Diff_\PP(M) \to 0.  
$$
We will use the following alternative description of $\PP(M)$;
see~\cite{Cer70}. Denote 
by $\EE(M)$ the space of all smooth functions $f:I\times M\to I$
without critical points and satisfying $f(r,x)=r$ on $\Op(\p I\times
M)$. We have a homotopy equivalence 
$$
   \PP(M)\to\EE(M),\qquad F\mapsto p\circ F,
$$
where $p:I\times M\to I$ is the projection. A homotopy inverse is
given by fixing a metric and sending $f\in\EE(M)$ to the unique
diffeomorphism $F$ mapping levels of $f$ to levels of $p$ and
gradient trajectories of $f$ to straight lines $I\times\{x\}$. Note that
the last map in the homotopy exact sequence
$$
   \cdots\to\pi_0\Diff(I\times M,\p I\times M)\to \pi_0\EE(M)\to
   \pi_0\Diff_\PP(M)
$$
associates to $f\in\EE(M)$ the flow from $0\times M$ to $1\times M$
along trajectories of a gradient-like vector field (whose isotopy
class does not depend on the gradient-like vector field). 

For the symplectic version of the pseudo-isotopy spaces, it will be
convenient to replace 
$I\times M$ by $\R\times M$ as follows: We replace $\EE(M)$ by the
space of functions $f:\R\times M\to \R$ without critical points and
satisfying $f(r,x)=r$ outside a compact set; $\Diff(I\times M,\p
I\times M)$ by the space $\Diff_c(\R\times M)$ of
diffeomorphisms that equal the identity outside a compact set; and
$\PP(M)$ by the space of diffeomorphisms of $\R\times M$ that equal
the identity near $\{-\infty\}\times M$ and have the form
$(r,x)\mapsto(r+f(x),g(x))$ near $\{+\infty\}\times M$. The
last map in the exact sequence
$$
   \cdots\to\pi_0\Diff_c(\R\times M)\to \pi_0\EE(M)\to
   \pi_0\Diff_\PP(M)
$$
then associates to $f\in\EE(M)$ the flow from $\{-\infty\}\times M$ to 
$\{+\infty\}\times M$ along trajectories of a gradient-like vector
field which equals $\p_r$ outside a compact set. 

We endow the spaces $\PP(M)$, $\EE(M)$ and $\Diff_c(\R\times M)$ 
with the topology of uniform $C^\infty$-convergence on $\R\times M$  
(and {\em not} the topology of uniform $C^\infty$-convergence on
compact sets), with respect to the product of the Euclidean metric on
$\R$ and any Riemannian metric on $M$. In other words, a sequence
$F_n\in\PP(M)$ converges to $F\in\PP(M)$ if and only if
$\|F_n-F\|_{C^k(\R\times M)}\to 0$ for every $k=0,1,\dots$.  
For example, consider any non-identity element $F\in \PP(M)$ and the
translations $\tau_c(r,x)=(r+c,x)$, $c\in\R$, on $\R\times M$. Then 
the sequence $F_n:=\tau_n\circ F\circ\tau_{-n}$ {\it does not
  converge} as $n\to\infty$ to the identity in $\PP(M)$, although it does
converge uniformly on compact sets. 
With this topology, the obvious inclusion maps from the spaces on
$I\times M$ to the corresponding spaces on $\R\times M$ are weak
homotopy equivalences. 
\medskip

\begin{remark} 
It was proven by Cerf in~\cite{Cer70} that $\pi_0\PP(M)$ is trivial
if $\dim M\geq 5$ and $M$ is simply connected. In the non-simply
connected case and for $\dim M\geq 6$, Hatcher and Wagoner
(\cite{HatWag73}, see also~\cite{Igu88}) have expressed $\pi_0\PP(M)$
in terms of algebraic K-theory of the group ring of $\pi_1(M)$. In
particular, there are many fundamental groups for which $\pi_1\PP(M)$
is not trivial. 
\end{remark}

\bigskip

Let us now fix a contact manifold $(M^{2n-1},\xi)$ and denote by
$(SM,\lambda_\st)$ its symplectization with its canonical Liouville
structure $(\om_\st=d\lambda_\st,X_\st)$.  
Any choice of a contact form $\alpha$ for $\xi$ yields an
identification of $SM$ 
with $\R\times M$ and the Liouville structure $\lambda_\st=e^r\alpha$,
$\om_\st=d\lambda_\st$, $X_\st=\p_r$. However, the following
constructions do not require the choice of a contact form. We will
refer to the two ends of $SM$ as $\{\pm\infty\}\times M$. 

We define the group of 
{\em symplectic pseudo-isotopies} 
\index{symplectic!pseudo-isotopy} 
of $(M,\xi)$ as
\begin{multline*}
   \PP(M,\xi) := \{F\in\Diff(SM)\mid 
   F^*\om_\st=\om_\st,\ 
   F=\id \text{ near }\{-\infty\}\times M,\\   
   F^*\lambda_\st=\lambda_\st \text{ near
   }\{+\infty\}\times M\}.  
\end{multline*}
Moreover, we introduce the space
\begin{multline*}
   \EE(M,\xi) := \{(\lambda,\phi) \text{ Weinstein structure on SM
     without critical points }\mid \\
   d\lambda=\om_\st,\ 
   (\lambda,\phi)=(\lambda_\st,\phi_\st) \text{ outside a compact
     set}\}
\end{multline*}
and its image $\bar\EE(M,\xi)$ under the projection
$(\lambda,\phi)\mapsto\lambda$. 
We endow the spaces $\PP(M,\xi)$, $\EE(M,\xi)$, and $\bar\EE(M,\xi)$
with the topology of uniform $C^\infty$-convergence on $SM=\R\times M$
as explained above.

\begin{lemma}
The map
$$
   \EE(M,\xi)\to\bar\EE(M,\xi),\qquad
   (\lambda,\phi)\mapsto\lambda,
$$
is a homotopy equivalence and the map 
$$
   \PP(M,\xi) \to \bar\EE(M,\xi),\qquad F\mapsto F^*\lambda_\st, 
$$
is a homeomorphism. 
\end{lemma}

\begin{proof}
The first map defines a fibration whose fiber over $\lambda$ is
the contractible space of Lyapunov functions for $X$ which are
standard at infinity. The inverse of the second map associates to
$\lambda$ the unique $F\in\Diff(SM)$ satisfying $F_*X=X_\st$ on $SM$
and $F=\id$ near $\{-\infty\}\times M$ (which implies
$F^*\lambda_\st=\lambda$ on $SM$). 
\end{proof}

Since $F\in\PP(M,\xi)$ satisfies $F^*\lambda_\st=\lambda_\st$ near
$\{+\infty\}\times M$, it descends there to a contactomorphism
$F_+:M\to M$. 
By construction, $F_+$ belongs to
the group $\Diff_\PP(M)$ of diffeomorphisms that are pseudo-isotopic
to the identity, so it defines an element in 
$$
   \Diff_\PP(M,\xi) := \{F_+\in\Diff_\PP(M)\mid F_+^*\xi=\xi\}. 
$$
Moreover, $F_+=\id$ if and only if $F$ belongs to the space
$$
   \Diff_c(SM,\om_\st) := \{F\in\Diff_c(SM)\mid
   F^*\om_\st=\om_\st\}
$$
of compactly supported symplectomorphisms of $(SM,\om_\st)$. Thus we
have a fibration
$$
   \Diff_c(SM,\om_\st)\to \PP(M,\xi)\to \Diff_\PP(M,\xi).
$$
The corresponding homotopy exact sequence fits into a commuting
diagram 
\begin{equation}\label{eq:pseudo}
\begin{CD}
   \pi_0\Diff_c(SM,\om_\st) @>>> \pi_0\PP(M,\xi) @>>>
   \pi_0\Diff_\PP(M,\xi) @>>> 0\\
   @VVV @VVV @VVV \\
   \pi_0\Diff_c(\R\times M) @>>> \pi_0\PP(M) @>>>
   \pi_0\Diff_\PP(M) @>>> 0,
\end{CD}
\end{equation}
where the vertical maps are induced by the obvious inclusions. 

The following is the main result of this section. 

\begin{thm}\label{thm:pseudo-symp}
For any closed contact manifold $(M,\xi)$ of dimension $2n-1\geq 5$,
the map $\pi_0\PP(M,\xi)\to\pi_0\PP(M)$ is surjective. 
\end{thm}

\begin{proof}
By the discussion above,
it suffices to show that the map $\pi_0\EE(M,\xi)\to\pi_0\EE(M)$
induced by the projection $(\lambda,\phi)\mapsto\phi$ is surjective.
So let $\psi\in\EE(M)$,
i.e., $\psi:\R\times M\to\R$ is a function without critical
points which agrees with $\phi_\st(r,x)=r$ outside a compact set
$W=[a,b]\times M$. 
By Theorem~\ref{thm:pseudo}, there exists a Morse homotopy
$\phi_t:\R\times M\to\R$ without critical points of index $>n$
connecting $\phi_0=\phi_\st$ with $\phi_1=\psi$ such that
$\phi_t=\phi_\st$ outside $W$ for all $t\in[0,1]$. 
We apply Theorem~\ref{thm:W-hom1} to the Weinstein cobordism
$\fW=(W,\om_\st,X_\st,\phi_\st)$ and the homotopy $\phi_t:W\to\R$. Hence
there exists a Weinstein homotopy $\fW_t=(W,\om_t,X_t,\phi_t)$, fixed
on $\Op\p_-W$ and fixed up to scaling on $\Op\p_+W$, such that
$\fW_0=\fW$.
Note that $\lambda_t=c_t\lambda_\st$ on
$\Op\p_+W$ for constants $c_t$ with $c_0=1$. So we can extend $\fW_t$
over the rest of $\R\times M$ by the function $\phi_\st$ and Liouville
forms $f_t(r)\lambda_\st$ such that $\fW_t=\fW$ on
$\{r\leq a\}$ and on $\{r\geq c\}$ for some sufficiently large $c>b$. 
By Moser's stability theorem,
we find a diffeotopy
$h_t:SM\to SM$ with $h_0=\id$, $h_t=\id$ outside $[a,c]\times M$, and
$h_t^*\fW_t=\fW$. Thus $h_1^*\fW_1=(\lambda,\phi)$ with the function
$\phi:=\psi\circ h_1$ and a Liouville form $\lambda$ which agrees with
$\lambda_\st$ outside $[a,c]\times M$ and satisfies
$d\lambda=\om_\st$. Hence 
$(\lambda,\phi)\in\EE(M,\xi)$ and $\phi$ is homotopic (via $\psi\circ
h_t$) to $\psi$ in $\EE(M)$, i.e., $[\phi]=[\psi]\in\pi_0\EE(M)$. 
\end{proof}

By Theorem~\ref{thm:pseudo-symp}, the second vertical map in the
diagram~\eqref{eq:pseudo} is surjective and we obtain

\begin{cor}
Let $(M,\xi)$ be a closed contact manifold of dimension $2n-1\geq
5$. Then every diffeomorphism of $M$ that is pseudo-isotopic to the
identity is smoothly isotopic to a contactomorphism of $(M,\xi)$. 
\end{cor}

\begin{remark}
Considering in the diagram~\eqref{eq:pseudo} elements in $\pi_0\PP(M)$
that map to $\id\in\pi_0\Diff_\PP(M)$, we obtain the following (non-exclusive)
dichotomy for a contact manifold $(M,\xi)$ of dimension $\geq 5$ for
which the map $\pi_0\Diff_c(\R\times M)\to\pi_0\PP(M)$ is nontrivial: 
Either there exists a contactomorphism of $(M,\xi)$ that is smoothly
but not contactly isotopic to the identity; or there exists a
compactly supported symplectomorphism of $(SM,\om_\st)$ which
represents a nontrivial smooth pseudo-isotopy class in $\PP(M)$. 
Unfortunately, we cannot decide which of the two cases occurs. 
\end{remark}

\subsection{Equidimensional symplectic embeddings of flexible
  Weinstein manifolds}

Finally, let us mention a recent result concerning equidimensional
symplectic embeddings of flexible Weinstein manifolds. Its proof goes
beyond the methods discussed in this paper. 

\begin{thm}[\cite{EliMur13}]\label{thm:self-embed}
Let $(W,\om,X,\phi)$ be a flexible Weinstein domain with Liouville
form $\lambda$. Let $\Lambda$ be any other Liouville form on $W$ such
that the symplectic forms $\om$ and $\Omega:=d\Lambda$ are
homotopic as non-degenerate (not necessarily closed) $2$-forms. 
Then there exists an isotopy $h_t:W\into W$ such that $h_0=\Id$ and
$h_1^*\Lambda=\eps\lambda+dH$ for some small $\eps>0$ and some smooth
function $H:W\to\R$. In particular, $h_1$ defines a symplectic
embedding $(W,\eps\om)\into(W,\Omega)$. 
\end{thm}

\begin{cor}[\cite{EliMur13}]\label{cor:flexible-embed}
Let $(W,\om,X,\phi)$ be a flexible Weinstein domain and $(X,\Omega)$
any symplectic manifold of the same dimension. Then any smooth
embedding $f_0:W\into X$ such that the form $f_0^*\Omega$ is exact and
the differential $df:TW\to TX$ is homotopic to a symplectic
homomorphism is isotopic to a symplectic embedding $f_1:(W,\eps\om)\into
(X,\Omega)$ for some small $\eps>0$. 
Moreover, if $\Omega=d\Lambda$, then the embedding $f_1$ can be chosen
in such a way that the 1-form $f_1^*\Lambda-i_X\om$ is exact. If, moreover,
the Liouville vector field dual to $\Lambda$ is complete, then
the embedding $f_1$ exists for arbitrarily large constant
$\eps$. \qed 
\end{cor}

\bigskip
{\bf Acknowledgement.} Part of this paper was written when the second author visited the Simons Center for Geometry and Physics at Stony Brook. He thanks the center for  the hospitality.
 

\end{document}